\documentclass[reqno]{amsart}

\usepackage{graphicx,pdfsync,amssymb, latexsym,amsfonts,amsbsy,color,subcaption,esint,mathtools,enumerate,nicefrac}
\usepackage{hyperref}
\hypersetup{
colorlinks=true,
linkcolor=blue,
filecolor=blue,      
urlcolor=blue,
citecolor=blue,
}

\usepackage{float}
\restylefloat{table}

\usepackage{enumitem}
\setitemize{itemsep=0.5em}
\setenumerate{itemsep=0.5em}

\usepackage[utf8]{inputenc}% for special characters
\usepackage[T1]{fontenc}% for special characters

\usepackage[top=1.1in, bottom=1in, left=1in, right=1in]{geometry}
\DeclareMathOperator{\Supp }{supp}
\DeclareMathOperator{\Id }{Id} 
 
\DeclareMathOperator{\D}{div}
\DeclareMathOperator{\Tr}{Tr} 
\DeclareMathOperator{\dist}{dist}

\DeclareMathOperator{\Lin}{lin} 
\DeclareMathOperator{\Osc}{osc} 
\DeclareMathOperator{\Tem}{tem}

\DeclareMathOperator{\Far}{far}

\DeclareMathOperator{\Cor}{cor}

\newtheorem{theorem}{Theorem}[section]
\newtheorem{lemma}[theorem]{Lemma}
\newtheorem{proposition}[theorem]{Proposition}
\newtheorem{definition}[theorem]{Definition}

\newtheorem{remark}[theorem]{Remark}
\newtheorem{conjecture}[theorem]{Conjecture}

\newcommand{\thistheoremname}{}
\newtheorem{genericthm}[theorem]{\thistheoremname}

 \newtheorem*{genericthm*}{\thistheoremname}
\newenvironment{namedthm*}[1]
  {\renewcommand{\thistheoremname}{#1}%
   \begin{genericthm*}}
  {\end{genericthm*}}

\usepackage{times}
\usepackage{setspace}
\def \TT  {\mathbb{T}} %torus
\def \RR {\mathbb{R}}  %real numbers
\def \NN {\mathbb{N}}  %natural numbers
\def \ZZ {\mathbb{Z}}  %integer numbers

\def \p {\partial}

\def \ep {\varepsilon}

\def \l {\lambda}
\def \L {\Lambda}

\def \ek {\mathbf{e}_{k}}

\def \bwk {\mathbf{ W}_{k}}
\def \bw {\mathbf{ W}}
\def \bok {\mathbf{\Omega}_k}

\numberwithin{equation}{section}

\begin{document}

\title[Sharp nonuniqueness for NSE]{Sharp nonuniqueness  for the Navier-Stokes equations}

\author{Alexey Cheskidov}
\address[Alexey Cheskidov]{Department of Mathematics, Statistics and Computer Science,
University of Illinois At Chicago, Chicago, Illinois 60607 and School of Mathematics, Institute for Advanced Study, 1 Einstein Dr., Princeton, NJ 08540, USA.}

\email{acheskid@uic.edu}

\author{Xiaoyutao Luo}

\address[Xiaoyutao Luo]{Department of Mathematics, Duke University, Durham, NC 27708 and School of Mathematics, Institute for Advanced Study, 1 Einstein Dr., Princeton, NJ 08540, USA.}

 \email{xiaoyutao.luo@duke.edu}

%    General info
%\subjclass[2020]{76D03, 35Q35}

\date{\today}

\begin{abstract}
In this paper, we prove a sharp nonuniqueness result for the incompressible Navier-Stokes equations in the periodic setting.  In any dimension $d \geq 2$ and given any $ p<2$, we show the nonuniqueness of weak solutions in the class $L^{p}_t L^\infty$, which is sharp in view of the classical Ladyzhenskaya-Prodi-Serrin criteria. The proof is based on the construction of a class of non-Leray-Hopf weak solutions. More specifically, for any $ p<2$, $q<\infty$, and $\varepsilon>0$, we construct non-Leray-Hopf weak solutions  $ u \in L^{p}_t L^\infty \cap L^1_t W^{1,q}$ that are smooth outside a set of singular times with Hausdorff dimension less than $\varepsilon$. As a byproduct,  examples of anomalous dissipation in the class $L^{ {3}/{2} - \varepsilon}_t C^{ {1}/{3}} $ are given in both the viscous and inviscid case.

\end{abstract}
\maketitle

%%%%%%%%%%%%%%%%%%%%%%%%%%%%%%%%%%%%%%%%%%%%%%%%%%%%%%%%%%%%%%%%%%%%%%%%%%%%%%%%%%%%%%%%%%%%%%%%%%%%%%%%%%%%%%%%%%%%%%%%%%%%%%
\section{Introduction}\label{sec:intro}
%%%%%%%%%%%%%%%%%%%%%%%%%%%%%%%%%%%%%%%%%%%%%%%%%%%%%%%%%%%%%%%%%%%%%%%%%%%%%%%%%%%%%%%%%%%%%%%%%%%%%%%%%%%%%%%%%%%%%%%%%%%%%%

\subsection{The incompressible Navier-Stokes equations}
The Navier-Stokes equations are a fundamental mathematical model of incompressible viscous fluid flow, written as
\begin{equation}\label{eq:NSE}
\begin{cases}
\p_t u -   \Delta u + \D( u \otimes u) + \nabla p = 0 &\\
\D u = 0,
\end{cases}
\end{equation}
posed on a spatial domain $\Omega \subset \RR^d$ with suitable boundary conditions. In \eqref{eq:NSE}, $u :[0,T]\times \Omega \to \RR^d $ is the unknown velocity, and $p :[0,T]\times \Omega \to \RR  $ is a scalar pressure. We consider the Cauchy problem of \eqref{eq:NSE} on a time interval $[0,T]$ for some initial data $u_0$ and $T>0$. 

We confine ourselves to the periodic case $\Omega = \TT^d  =\RR^d /\ZZ^d $ in dimension $d \geq 2$ and  consider solutions with zero spacial mean
$$
\int_{\TT^d} u(t,x) \, dx = 0,  
$$
which is propagated under the evolution of the equation \eqref{eq:NSE}.

In this paper, we study the question of uniqueness/nonuniqueness for weak solutions of \eqref{eq:NSE}. The notion of weak solutions refers to that of distributional solutions which solve \eqref{eq:NSE} in the sense of space-time distribution with minimal regularity, cf. \cite{MR316915,MR2838337}.

\begin{definition}\label{def:weak_solutions}
Denote by $\mathcal{D}_T$ the space of divergence-free test function $\varphi \in C^\infty (\RR \times \TT^d ) $ such that  $\varphi =0$ if $t\geq T$.

Let $ u_0 \in L^2(\TT^d)$  be weakly divergence-free\footnote{It is possible to consider more general initial data, such as $L^p $ for some $ 1 \leq p \leq \infty$ as in~\cite{MR316915}.}. A vector field $ u \in L^2 ( [0,T] \times \TT^d)$ is a weak solution of \eqref{eq:NSE} with initial data $u_0$ if the followings hold:
\begin{enumerate}
    \item For $a.e.$ $t\in [0,T]$, $u$ is weakly divergence-free;
    
    \item For any $\varphi \in \mathcal{D}_T$,
\begin{equation}
\int_{\TT^d} u_0(x)\cdot \varphi(0,x ) \, dx = - \int_0^T \int_{\TT^d} u\cdot \big(  \partial_t \varphi+ \Delta \varphi +  u \cdot \nabla \varphi  \big) \, dx dt .
\end{equation}
\end{enumerate}

\end{definition}

In the literature, such solutions are sometimes called ``very weak solutions'' \cite{MR1755865,MR1798753} due to the minimal regularity assumptions of only being square integrable in space-time.  Remarkably, by \cite[Theorem 2.1]{MR316915}, up to possibly redefining $u$ on a set of measure zero in space-time, the above weak formulation is equivalent to the integral equation 
\begin{equation}\label{eq:NSE_integral_formulation}
u = e^{t\Delta}u_0 + \int_0^t e^{(t-s)\Delta} \mathbb{P}\D(u\otimes u) (s)\, ds ,
\end{equation}
where $ e^{t\Delta}$ is the heat semigroup and $\mathbb{P}$ is the Leray projection onto the divergence-free vector fields. Note that the formulation \eqref{eq:NSE_integral_formulation} was also used in a variety of works \cite{MR166499,MR760047,MR1808843} to construct unique solutions (called mild solutions) of \eqref{eq:NSE} when the initial data $u_0$ is critical or subcritical, starting from the work of Fujita  and Kato \cite{MR166499}.

A more physical class of weak solutions, introduced by Leray \cite{MR1555394} and constructed by Leray \cite{MR1555394} in $\RR^3$ and Hopf \cite{doi:10.1002/mana.3210040121} in general domains in $d\geq 2$, is also considered in the literature. 
\begin{definition}\label{def:LHweak_solutions}
	A  weak solution  $u$  of \eqref{eq:NSE} is called Leray-Hopf weak solution if $u \in C_w([0,T]; L^2(\TT^d)) \cap L^2(0,T; H^1(\TT^d))$ and
	\begin{equation}\label{eq:energy_inequality}
	\frac{1}{2} \|u(t) \|_2^2 + \int_0^t\|\nabla u (s)\|_2^2 \,ds \leq \frac{1}{2} \|u(0) \|_2^2,
	\end{equation}
for all $t\in[0,T]$
\end{definition}
The Leray-Hopf weak solutions encode the natural conservation law of \eqref{eq:NSE} regarding the two important physical quantities, the kinetic energy and the dissipation of the energy, and satisfy much better properties than general weak solutions, especially in the most relevant case of 3D, such as Leray's structure theorem~\cite{MR1555394}, weak-strong uniqueness~\cite{MR0136885,MR2237686,MR2838337}, partial regularity~\cite{MR933230,MR673830}, integrability of higher Sobolev norms~\cite{MR3258360}, estimates of potential blowup rates~\cite{MR1992563,MR3475661,MR3465978}. In fact, it is well-known that Leray-Hopf solutions are smooth and unique in 2D.  These nice properties, much desirable from a regularity stand point, make it significantly harder to construct nonunique Leray-Hopf solutions in $d\geq 3$, though partial results~\cite{Lady_enskaja_1969,MR3341963} and numerical evidence~\cite{1704.00560} are available.

The purpose of this paper is to produce sharp counterexamples to the classical Ladyzhenskaya-Prodi-Serrin uniqueness of the Navier-Stokes equations using the convex integration technique. On one hand, while the solutions constructed in this paper live on a borderline of a class of Leray-Hopf solutions, they have unbounded energy globally on $[0,T]$ and do not satisfy \eqref{eq:energy_inequality}. On the other hand, both the energy $\|u(t)\|_2^2 $ and energy dissipation rate $\| \nabla u(t)\|_2^2$ of these solutions are finite not only locally but also in  certain time-averaged senses as well. The latter concept of taking time ensemble of physical quantities is especially essential in Kolmogorov’s theory of turbulence and is measurable experimentally. Nevertheless, we consider a wider class of weak solutions which is very natural from the mathematical point of view, and prove that the classical $L^2_tL^\infty$ uniqueness of Ladyzhenskaya-Prodi-Serrin (see Theorem \ref{thm:FJR_uniqueness}) in this class is sharp. To our knowledge, this is the first sharp counterexample to the classical uniqueness results for the Navier-Stokes equations. It is obviously an open problem whether such a sharp nonuniqueness can be extended to weak solutions with bounded energy or even Leray-Hopf weak solutions.

Since the submition of this paper, in remarkable work~\cite{2112.03116} Albritton, Brué, and Colombo proved that Leray-Hopf solutions are not unique for a forced Navier-Stokes equations, that is \eqref{eq:NSE} with a forcing term on the right-hand side. Their method is completely different from us and is relatd to Vishik's unstable vortex~\cite{Vishik1,Vishik2}. The nonuniqueness of Leray-Hopf solutions of the unforced Navier-Stokes equations remains open.

%%%%%%%%%%%%%%%%%%%%%%%%%%%%%%%%%%%%%%%%%%
\subsection{The Ladyzhenskaya-Prodi-Serrin threshold}
%%%%%%%%%%%%%%%%%%%%%%%%%%%%%%%%%%%%%%%%%%

We first discuss the uniqueness results in our context. Since the existence of weak solutions was known, there has been a vast body of literature on the uniqueness of weak solutions. For brevity, we do not distinguish the underlying spatial domain, assumptions on the external forces, and smoothness of the initial data for the results that we are going to mention. Instead, our discussion would focus on the core ideas and the scaling threshold. The functional setup is the mixed Lebesgue space $L^p_t L^q = L^p(0,T;L^q(\Omega))$, where $\Omega$ is the spatial domain, such as the whole space $\RR^d$ or a bounded domain.

For Leray-Hopf weak solutions,  Prodi~\cite{Prodi1959}, Serrin~\cite{MR0136885}, and  Ladyzhenskaya~\cite{MR0236541} proved that if a Leray-Hopf solution $u$  satisfies
$$
u \in L^p_t L^q \quad \text{for some $p <\infty$ and $q> d$ such that $\frac{2}{p} + \frac{d}{q} \leq 1$},
$$
then \emph{all} Leray-Hopf solutions with the same initial data must coincide. The the endpoint case $(p, q) = (\infty , d)$ was treated much later in \cite{MR1403221} after the attempts \cite{MR767409,MR762786}. This type of results is often referred to as \emph{weak-strong uniqueness} meaning that if there exists a strong solution, then any weak solution with the same initial data coincides with it. In fact, a membership in such functional classes implies the regularity of Leray-Hopf solutions as well, though we will not go into details in this direction and simply refer interested readers to \cite{MR0136885,MR1992563,MR3475661} and references therein. In what follows, we will refer to all the above $L^p_t L^q$ conditions with scaling $\frac{2}{p} + \frac{d}{q} = 1 $ as the Ladyzhenskaya-Prodi-Serrin criteria/threshold for both the uniqueness and regularity.

For general weak solutions, one loses the property of weak-strong uniqueness due to the lack of the energy inequality \eqref{def:LHweak_solutions}. For such weak solutions, one can instead study the uniqueness issue within certain functional classes. Due to technical reasons, we have to work with a smaller class of functional space $C_t  L^d$ rather than the standard $L^\infty_t L^d$ space at the endpoint $(p,q) = (\infty ,d )$. To simplify notations, given $p,q \in [1,\infty]$ let us denote the Banach space
\begin{equation*}
X^{p,q}([0,T] ; \TT^d)  = 
\begin{cases}
 L^p(  0,T; L^q( \TT^d) )  & \text{if $p \neq \infty$, }  \\
  C(  [0,T]; L^q( \TT^d) )  & \text{if $p = \infty$. }
\end{cases}
\end{equation*}

The general scheme of proving uniqueness for general weak solutions is to recast \eqref{eq:NSE_integral_formulation} into the following abstract integral formulation
\begin{equation}\label{eq:NSE_abstract}
u = e^{t\Delta}u_0 + B(u, u),
\end{equation}
and study the continuity of the bilinear operator $B$ in the various underlying functional spaces, see for instance~\cite{MR3469428} and references therein. The first result in this direction dates back to Fabes, Jones, and Rivi\`ere \cite{MR316915} who proved that weak solutions in the class $ X^{p,q} $ are unique if $ \frac{2}{p} + \frac{d}{q} \leq 1$ and  and $d < q  <\infty  $. In other words, the Ladyzhenskaya-Prodi-Serrin criteria hold for just $L^2_{t,x} $ weak solutions. The limit case $(p , q ) = ( \infty , d) $  was later covered first by \cite{MR1813331} for $d  = 3$, and then in \cite{MR1724946,MR1680809,MR1876415} via different methods and for different spatial domains. Interestingly, in dimensions $d \geq 4$, the endpoint case $(p , q ) = ( \infty , d) $ can be strengthened to $L^\infty_t L^d$~\cite{MR1876415}.

Based on a scaling analysis, when $\frac{2}{p} + \frac{d}{q} = 1 $, the spaces $X^{p,q}$ are invariant under the parabolic scaling of the equation $ u \mapsto u_\l := \l u(\l^2 t, \l x)$. In the literature, the space $X^{p,q}$ is called sub-critical when $\frac{2}{p} + \frac{d}{q} < 1 $, critical  when  $\frac{2}{p} + \frac{d}{q} = 1 $, and super-critical when  $\frac{2}{p} + \frac{d}{q} > 1 $. The uniqueness holds in all the critical and sub-critical spaces $X^{p,q}$. Moreover, thanks to~\cite{MR316915,MR760047,MR1813331,MR1876415} when $\frac{2}{p} + \frac{d}{q} \leq 1 $, all weak solutions (in the sense of Definition~\ref{def:weak_solutions}) belonging to the class $X^{p,q}$ are automatically Leray-Hopf\footnote{Our setting is on the $d$-dimensional torus and the initial data is always $L^2$.} and hence, by the  Ladyzhenskaya-Prodi-Serrin criteria, regular. In other words, within the scale of $X^{p,q}$ spaces, sub-critical or critical weak solutions are classical solutions.

We can summarize these uniqueness results  as follows. Since these results were originally stated for $\RR^d$, we also include a unified proof applicable to our specific setup in the appendix for readers' convenience.

\begin{theorem}[Ladyzhenskaya-Prodi-Serrin criteria]\label{thm:FJR_uniqueness}
Let $d \geq 2$ and $u$ be a weak solution of \eqref{eq:NSE} such that $u \in X^{p,q} $  for some $p, q \in [1,\infty]$ such that $ \frac{2}{p} + \frac{d}{q} \leq 1$.  Then
\begin{enumerate}
\item $u$ is unique in the class of $X^{p,q}$ weak solutions,
\item  $u$ is a Leray-Hopf solution, and regular on $(0,T]$. 
\end{enumerate}
\end{theorem}

So far, the positive results suggest the Ladyzhenskaya-Prodi-Serrin threshold $\frac{2}{p} + \frac{d}{q} = 1$ as the critical regularity threshold for uniqueness/nonuniqueness of the weak solutions. One would naturally ask what would happen in the super-critical regime $ \frac{2}{p} + \frac{d}{q} > 1$, or more specifically, whether the following conjecture is valid.

\begin{conjecture}\label{conject:main}
Let $d \geq 2$ and $p, q \in [1,\infty]$ such that $ \frac{2}{p} + \frac{d}{q} > 1$. Then 

\begin{enumerate}
    \item There exist two weak solutions $u, v \in  X^{p,q} $ of \eqref{eq:NSE} such that
 \[
 u(0) = v(0) \; \text{but}\;   v \ne u.
 \]

\item There exists a weak solution $u \in  X^{p,q} $ of \eqref{eq:NSE} such that $u$ is not Leray-Hopf.

\end{enumerate}

\end{conjecture}

In stark contrast to the positive result of Theorem \ref{thm:FJR_uniqueness}, which has been known for quite some time, Conjecture \ref{conject:main} was completely open until a very recent groundbreaking work \cite{MR3898708} of Buckmaster and Vicol. Following the breakthrough~\cite{MR3898708},  the nonuniqueness has been shown in dimension $d\geq 3$ under various settings:  $C_t L^{2+}$ \cite{MR3898708,1809.00600} in dimension $d= 3$ and $H^{1/200-}$\cite{MR3951691} in dimension $d \geq 4$. These works used a unified approach to tackle both parts of Conjecture \ref{conject:main} at the same time: one perturbs a given smooth solution of \eqref{eq:NSE} to obtain a ``wild'' solution with certain regularity $X^{p,q}$, and then the existence of such a wild solution implies both the nonuniqueness of weak solutions and the existence of non-Leray-Hopf solutions in the said class. Even though these works reveal that the nonuniqueness of weak solutions can emerge in the ``low'' regularity regime, we have to stress that the regularity of current nonunique weak solutions is far very from the critical threshold $ \frac{2}{p} + \frac{d}{q} =1$.

Besides the gap of the regularity between the critical threshold and the current nonuniqueness constructions, another important question left open since~\cite{MR3898708} is whether weak solutions in 2D are unique. Unfortunately, the strategy of \cite{MR3898708},  which was in turn developed from a series of works \cite{{MR3374958,MR3530360,MR3866888,1701.08678}}, breaks down in 2D. The reason is that roughly speaking, the framework of \cite{MR3898708} is $ L^2$ critical, in the sense that the mechanism can produce finite energy wild solutions if the system is $L^2$ super-critical. This heuristic has been confirmed in \cite{1809.00600,1808.07595,MR4097236} for the generalized Navier-Stokes equations with fractional dissipation.

Since the 2D Navier-Stokes equations is $L^2$-critical, there are no nonuniqueness results for the 2D case to date, and all known solutions are Leray-Hopf ones. In fact, a direct corollary of \cite{MR3898708} is false in 2D: any $C_t L^2$ weak solution of the 2D Navier-Stokes equations is Leray-Hopf, and hence smooth and unique. One of the main results in this paper is to show that the nonuniqueness of weak solutions holds even in 2D.

\begin{theorem}[Strong nonuniqueness in 2D]\label{thm:main_short}
Let $d=2$ be the dimension. For any divergence-free $u_0\in L^2(\TT^2)$, there exists a weak solution $u$ (in the sense of Definition \ref{def:weak_solutions}), different from the Leray-Hopf solution, and such that $\lim_{t\to 0+}\| u(t) -u_0\|_{L^2(\TT^2)} = 0$.
\end{theorem}

It is worth noting that the nonuniqueness is proved here in a stronger sense than~\cite{MR3898708,MR3951691}, namely that \emph{every} solution is nonunique in the class of weak solutions. We can classify different types of nonuniqueness results as follows. Here, $X$ denotes different functional classes of weak solutions.
\begin{itemize}
    \item ``Weak nonuniqueness'': there exists a nonunique weak solution in the class $X$.
    
    \item ``Strong nonuniqueness'': any weak solution in the class $X$ is nonunique.
\end{itemize}
Under this classification, currently the only strong nonuniqueness available is  \cite{1809.00600} for $3D$, where $X$ can be taken as $  C_t  H^{\ep}$  weak solutions with intervals of regularity for a small $\ep>0$. In fact, the main results of \cite{1809.00600} can be adapted to show the strong nonuniqueness of weak solutions on $\TT^d$ for dimension $d=4$ as well, since Leray-Hopf solutions in $d \leq 4$ have intervals of regularity.

The next result is our headline theorem, where we prove the strong nonuniqueness of weak solutions in a class of $ L^p_t L^\infty$, for any $p<2$ and in any dimension $d \geq 2$, which is sharp in view of the Ladyzhenskaya-Prodi-Serrin criteria. This theorem is a direct consequence of our main theorem, Theorem \ref{thm:main_thm_2} below, where a detailed list of the properties of constructed solutions can be found. In particular, we settle Conjecture \ref{conject:main} regarding the sharpness of Ladyzhenskaya-Prodi-Serrin criteria in the case $(p,q) =(2 ,\infty )$.

\begin{theorem}[Sharp nonuniqueness in $d\geq 2$]\label{thm:main_sharp_short}
Let $d \geq 2$ be the dimension and $1 \leq p< 2 $. 
\begin{enumerate}
    \item A weak solution  $u \in L^p(0,T; L^\infty(\TT^d))$  of \eqref{eq:NSE} is not unique in the class $   L^p(0,T; L^\infty(\TT^d))$ if   $u$ has at least one interval of regularity.

    \item There exist non-Leray-Hopf weak solutions $u \in L^p(0,T; L^\infty(\TT^d))$.
\end{enumerate}

\end{theorem}

\begin{remark}
A few remarks are in order:
\begin{enumerate}

\item In particular, one can apply the theorem to a smooth solution to obtain a nonunique weak solution in the class $ L^p_t L^\infty$. In fact, we can alternatively present the result as follows: for any smooth initial data, there are infinitely many weak solutions with regularity $L^p_t L^\infty$ that coincide with the unique local smooth solution for a short time, see the proof of Theorem \ref{thm:main_euler_onsager} at the end of Section \ref{sec:outline}.
    
    \item Even though general $L^2_{t,x}$ weak solutions only attain their initial data in the distributional sense, our nonunique solutions can attain  their initial data strongly in $L^2$ or even in the classical sense. In other words, the mechanism of nonuniqueness does not stem from the roughness of the initial data, but rather the persistence of space-time oscillations.

    \item The nonunique weak solutions are obtained by modifying a given solution on an interval of regularity. As a result, these solutions can be arranged to only differ from a given Leray-Hopf weak solution on a fixed time interval, and hence they attain the initial data strongly in $L^2$ as the Leray-Hopf ones. This together with the Leray structure theorem implies that in dimensions $2\leq d \leq 4$, the Cauchy problem of \eqref{eq:NSE} has infinite many weak solutions for any divergence-free initial data $u_0 \in L^2$.

\item One clearly sees that Theorem \ref{thm:main_sharp_short} only implies the sharpness of the Ladyzhenskaya-Prodi-Serrin criteria $\frac{2}{p} + \frac{d}{q} \leq 1 $ at the endpoint $(p,q) = (2, \infty)$ and the rest of the borderline regime remains open. In fact, these nonunique weak solutions also live on the borderline of the Beale-Kato-Majda criterion, as we shall see in Theorem \ref{thm:main_thm_2} below. 

\item In view of Theorem \ref{thm:FJR_uniqueness}, nonunique solutions cannot live in the class $L^2_t L^\infty$ where weak solutions are Leray-Hopf, however, Theorem \ref{thm:main_sharp_short} shows the existence of nonunique solutions on the borderline of this Leray-Hopf class and raises the question of whether such constructions can be extended to Leray-Hopf solutions of lower regularity.
\end{enumerate}
\end{remark}

%%%%%%%%%%%%%%%%%%%%%%%%%%%%%%%%%%%%%%%%%%
\subsection{The main theorem and intervals of regularity}
%%%%%%%%%%%%%%%%%%%%%%%%%%%%%%%%%%%%%%%%%%
We now present the main theorem of the paper, which implies immediately Theorems \ref{thm:main_short} and \ref{thm:main_sharp_short} above as we show in Section \ref{sec:outline}. One of the most interesting features of the constructed weak solutions is that they possess intervals of regularity, i.e. they are classical solutions on many sub-intervals whose union occupies a majority of the time axis, cf. the classical Leray-Hopf solutions \cite{MR1555394} and recent works~ \cite{1809.00600,2102.03244,2102.06085} on wild solutions with such a property, which we will discuss in detail towards the end of the introduction.  

\begin{theorem}[Main theorem]\label{thm:main_thm_2}
Let $d \geq 2$ be the dimension and $1\leq p <2$, $q < \infty$, and $\ep>0$. For any smooth, divergence-free vector field $v \in C^\infty ([0,T] \times \TT^d)$ with zero spatial mean for each $t \in[0,T]$, there exists a weak solution $u $ of \eqref{eq:NSE} and a set 
$$
\mathcal{I} = \bigcup_{i=1}^\infty (a_i ,b_i) \subset [0,T],
$$
such that the following holds.
\begin{enumerate}
\item The solution $u$ satisfies
$$ 
u \in L^p(0,T; L^\infty(\TT^d) )  \cap L^1(0,T; W^{1,q}(\TT^d) ).
$$
\item $u $ is a smooth solution on $(a_i , b_i )$ for every $i$, namely
$$
u|_{\mathcal{I}\times \TT^d }\in C^\infty(\mathcal{I} \times \TT^d).
$$
In addition, $u$ agrees with the unique smooth solution with the initial data $v(0)$ near $t=0$ and is also regular neat $t=T$.

\item The Hausdorff dimension of the residue set $\mathcal{S}= [0,T] \setminus \mathcal{I}$ satisfies
$$
d_{\mathcal{H}} (\mathcal{S} ) \leq \ep.
$$

\item The solution $u$ and the given vector field $v$ are $\ep$-close in $L_t^p L^\infty \cap L_t^1 W^{1,q} $:
$$
\| u  -v \|_{ L^p(0,T; L^\infty(\TT^d)) \cap L^1(0,T; W^{1,q}(\TT^d) )  }   \leq \ep. 
$$
\end{enumerate}

\end{theorem}

\begin{remark}\label{remark:main_thm}
We list a few remarks here concerning the main theorem.
\begin{enumerate}

 \item In terms of the scaling, $u$ also lies on the borderline of the Beale-Kato-Majda criterion \cite{MR763762} which scales as $L^1_t W^{1, \infty}$. Even though it is not known whether the Beale-Kato-Majda criterion implies the uniqueness in the setting of $L^2_{t,x}$ weak solutions, this in a sense suggests that one can not beat the scaling.

 \item The residue set $S=[0,T] \setminus \mathcal{I}$ is a singular set in the sense that for any $t \not \in S$, there is $\delta>0$ such that $u \in C^\infty((t-\delta, t+\delta)\times \TT^d )$. In addition, we do not prove $d_{\mathcal{H}} (\mathcal{S} ) >0 $, but it follows that the set $\mathcal{S}$ is at least nonempty from our current construction. It seems possible that a more refined bookkeeping would allow to show a similar lower bound on the Hausdorff dimension of $\mathcal{S}$ as well.

 \item Since the solution $u(t)$ is smooth on $ (a_i,b_i)$,  the energy equality is satisfied
$$
\frac{1}{2} \|u(t_1) \|_2^2 + \int_{t_0}^{t_1} \|\nabla u (s)\|_2^2 \,ds  = \frac{1}{2} \|u(t_0) \|_2^2 \quad \text{for all $t_0, t_1 \in (a_i ,b_i)$}.
$$
However, the energy equality on $[0,T]$ is not valid, which can be seen by taking a vector field $v$ with an increasing energy profile.

  \item The driving mechanism of nonuniqueness is a result of large chunks of mass emerging from/escaping to finer time scales. There is no blowup on each interval of regularity $(a_i,b_i)$ but norms do blow up as $i \to \infty$, at least for higher order norms.

\end{enumerate}

\end{remark}

%%%%%%%%%%%%%%%%%%%%%%%%%%%%%%%%%%%%%%%%%%
\subsection{Applications to the Euler equations}
%%%%%%%%%%%%%%%%%%%%%%%%%%%%%%%%%%%%%%%%%%

Our results also apply to the inviscid case with no changes in contrast to a recent work \cite{1809.00600}, which relies heavily on the parabolic regularization.

\begin{theorem}\label{thm:main_euler}
Theorem \ref{thm:main_thm_2} also holds for the Euler equations. Namely, under the same assumptions, there exists a weak solution $u$ of the Euler equations satisfying the same properties.
\end{theorem}

As a byproduct of the construction, we provide improvements and extensions to the Onsager conjecture in the negative direction, where the best result in 2D currently stands at $L^\infty_t C^{1/5 -\ep}$~ \cite{MR3374958}, see also \cite[pp. 1817]{MR3987721} and \cite{novack2018nonuniqueness}.

\begin{theorem}\label{thm:main_euler_onsager}
Let $d \geq 2$ be the dimension and $\ep>0$. For the Euler equations or the Navier-Stokes equations on $\TT^d$, there exist infinitely many non-conserving weak solutions $u \in L^{\frac{3}{2} -\ep}_t C^{\frac{1}{3}} \cap L^1_t C^{1 -\ep}$ with the same initial data. 
\end{theorem}

Theorem \ref{thm:main_euler_onsager} directly follows from Theorem \ref{thm:main_thm_2} via the interpolation $ L^{\frac{3}{2} -}_t C^{\frac{1}{3}} \subset L^{p}_t L^\infty \cap L^{1}_t C^{1-}  $ and the embedding $L^{1}_t W^{1,q} \subset   L^{1}_t C^{1-} $ in both the viscous case and inviscid case.
\begin{remark}
We list a few remarks for Theorem \ref{thm:main_euler_onsager}.
\begin{enumerate}
\item A simple argument shows that such non-conservative Euler solutions can arise in the vanishing viscosity limit of the Navier-Stokes equations (using the same constructions). In particular, these solutions belongs to $L^1_t H^1 \subset L^1_t W^{1,q} $, and hence the energy dissipation rate is finite in $L^{1/2}_t$.

    \item This result appears to establish the first non-conserving solutions with an exact ``$\frac{1}{3}$-H\"older regularity'' in space, albeit with a non-optimal $L^{3/2 -}_t$ exponent in time. The optimal exponent should be $L^{3-}_t C^{1/3} $ based on the positive results~\cite{MR1298949,MR2422377}.

    \item One of the reason for this non-optimal $L^{3/2 -}$ exponent is that the scaling of our solutions is not designed to produce sharp anomalous dissipation rates, It is also the first nonuniqueness result for the Euler equations with ``nonuniqueness scaling'' (i.e. $L^1_t C^{1 -\ep} $) deviating significantly from its ``Onsager scaling'' (i.e. $ L^{\frac{3}{2} -\ep}_t C^{\frac{1}{3}}$) cf.~\cite{MR3210150,MR3614753,2004.00391,2101.09278}.

    \item Our solutions are highly oscillatory in space-time and hence not continuous on $[0,T]\times \TT^d$, which is in stark contrast to previous constructions \cite{MR3374958,MR3302631,MR3530360,MR3866888,1701.08678,2004.00391,2101.09278}. In fact, the kinetic energy of our solutions becomes unbounded in a piece-wise constant fashion: on each interval of regularity $(a_i, b_i) \subset \cup_i (a_i, b_i)$, the energy $\| u(t)\|_2^2$ is a constant, but $\sup_{t\in(a_i,b_i)}\| u(t)\|_2^2 \to \infty $ as $i \to \infty$.

\end{enumerate} 

\end{remark}

%%%%%%%%%%%%%%%%%%%%%%%%%%%%%%%%%%%%%%%%%%
\subsection{Main ideas of the construction}
%%%%%%%%%%%%%%%%%%%%%%%%%%%%%%%%%%%%%%%%%%
The construction in Theorem \ref{thm:main_thm_2} is based on an iterative scheme to obtain suitable approximate solutions to \eqref{eq:NSE}  that consists of two main steps. The first step is a concentration procedure for producing intervals of regularity while the second step uses the convex integration to finish the iteration. These two steps are completely independent of each other: one can skip the first step and only iterate with the convex integration scheme to obtain nonunique $L^p_t L^\infty$ weak solutions without any intervals of regularity.

The first step is to concentrate the stress error of the approximate solutions to many smaller sub-intervals, allowing us to achieve a small Hausdorff dimension of the singular times. One can consider this step as a temporally intermittent variant of the ``gluing technique'', a key ingredient in \cite{MR3866888} for the resolution of the Onsager conjecture which was temporally homogeneous. In particular, this ensures the final approximate solution is an exact smooth solution on many small intervals. This is done by adding a very small corrector to the existing approximate solution. More specifically, to obtain such a corrector, we first find the correctors on each small interval where it is designed to balance the stress error. Due to the local solvability of the Navier-Stokes or Euler equations, these correctors exist and are smooth provided the interval is sufficiently small. Then we use a partition of unity in time to glue these correctors on each small interval to obtain the corrector on $[0,T]$ which remains small, say in $L^\infty_t H^d$. The partition that we use has a very sharp transition near the ends of each interval, which effectively concentrates all the stress error to those regions. Crucially, adding such a small corrector to the existing solution keeps the size of the stress error unchanged $L^1$ in time up to a constant multiple. After this concentration procedure, the stress error is zero on a large subset of the time interval and thus the size of the concentrated stress error is much larger on its support set.

The second step uses a convex integration scheme to add another perturbation to the concentrated solution,   reducing the size of the stress error. This convex integration technique  has been developed over the last decades, see for instance \cite{MR2600877,MR3090182,MR3374958,MR3302631,1701.08678,MR3866888,MR3898708,MR3951691} and references therein, since its inception to fluid dynamics in \cite{MR2600877}. For the Navier-Stokes equations, this typically consists of adding carefully designed velocity perturbation so that the nonlinear interaction balance the stress error in a suitable sense. In particular, its latest iteration for the transport equation in \cite{2004.09538} allows us to achieve a very high level of temporal concentration of the perturbation in the sense that higher Sobolev norms blow up while their time averages remain bounded. This is done by adding in temporal oscillations that are also highly intermittent in the velocity perturbation. The introduction of temporal concentration in the convex integration scheme allows us to trade temporal integrability for spatial regularity, answering a question raised in \cite[Problem 4.4]{BV2021}. In fact, the velocity perturbation oscillates much faster in time than in space, which is vital to obtain the sharp bounds $L^p_t L^\infty \cap L^1_t W^{1,q}$. To avoid a dimensional loss, the ``building blocks'' used in the scheme are almost spatially homogeneous, in stark contrast to \cite{MR3898708,MR3951691,1809.00600}.

The most difficult and important part of the iterative scheme is ensuring that the perturbation $w$  satisfies the regularity $w \in L^p_t L^\infty \cap L^1_t W^{1,q}$, while at the same time successfully reducing the size of the stress error. This boils down to balancing four different aspects of the perturbation: temporal and spatial oscillation/concentration. Intuitively, it is known that concentration may be used to trade integrability for derivative, while oscillation allows for gaining derivative with differential operators with negative order. In the present work, the leading order effects are temporal concentration and spatial oscillation which contribute most to getting the sharp regularity $L^p_t L^\infty \cap L^1_t W^{1,q}$, whereas temporal oscillation and spatial concentration effectively play minor roles. Among other small technical improvements, the temporal oscillation is a necessary part of the space-time convex integration scheme in \cite{2004.09538} and the spatial oscillation is used to get negligible interference between the building blocks especially in 2D. We refer to the discussion in Section \ref{subsec:osc_con} for more details.

%%%%%%%%%%%%%%%%%%%%%%%%%%%%%%%%%%%%%%%%%%
\subsection{Comparison with previous works}
%%%%%%%%%%%%%%%%%%%%%%%%%%%%%%%%%%%%%%%%%%

In the last part of the introduction, we compare our main results to the previous works and list a few open questions. We divide the discussion into three topics as follows.

\subsubsection*{Regularity threshold for uniqueness/nonuniqueness }
The first nonuniqueness result for the Navier-Stokes system was established in  \cite{MR3898708} by Buckmaster and Vicol, where finite energy nonunique weak solutions were constructed in 3D. Even though the nonuniqueness is only proved in $C_t L^2$, the iteration scheme in \cite{MR3898708} allows for a very small regularity $H^{\ep}$ for $\ep \ll 1$, which was then used in \cite{1809.00600} to show nonuniqueness at such a regularity. The work \cite{MR3951691} built upon the observation that in higher dimensions, weak solutions can be less intermittent, and thus the regularity of nonuniqueness was improved to $H^{1/200-}$ for $d\geq 4$. In fact, as noted in \cite{Taoblog}, in very high dimension one can show nonuniqueness in $C_t H^{1/2-}$ or $H^{1/2-}$ in the stationary case, although the regularity $H^{1/2 -}$ is still very far from the critical scale $  H^{ \frac{d-2}{2}} $ or $L^{d}$.

Below we compare different results using the scales of space-time Lebesgue spaces $X^{p,q}$. 

\begin{table}[H]
\begin{tabular}{|l|l|l|l|}
\hline
Results                    & Category & Scaling & Range \\ \hline
Leray-Hopf solutions      &   \small Existence   & \small$\frac{2}{p} + \frac{d}{q} = \frac{d}{2}$       & $  q  \geq 2   $ \\ \hline
\cite{MR3898708}   & \small Nonuniqueness       & \small$\frac{2}{p} + \frac{d}{q} = \frac{d}{2}  $        & $q = 2$  and $d=3$   \\ \hline
\cite{1809.00600}   & \small Nonuniqueness       & \small$\frac{2}{p} + \frac{d}{q} = \frac{d}{2} - \ep$        & $q = 2+$  and $d=3$   \\ \hline
\cite{MR3951691}   & \small Nonuniqueness       & \small$\frac{2}{p} + \frac{d}{q} = \frac{d}{2} - \frac{1}{200}$        & $q = 2+$  and $d \geq 4 $   \\ \hline
Theorem~\ref{thm:main_thm_2}     & \small Nonuniqueness        & \small$\frac{2}{p} + \frac{d}{q} = 1+ \ep $       & $ q = \infty $     \\ \hline
Theorem~\ref{thm:FJR_uniqueness} & \small Uniqueness & \small$ \frac{2}{p} + \frac{d}{q} = 1 $      & $   q \leq \infty $   \\ \hline
\end{tabular}
\end{table}

In light of the current state, we expect the nonuniqueness of weak solutions continue to hold in the full range of the super-critical regime $ \frac{2}{p} + \frac{d}{q} >1$. Unfortunately, the method developed in this paper heavily relies on the constraint $p<2$ ($q=\infty$) and is not able to achieve the nonuniqueness of weak solutions in $X^{p,q}$ for $p\geq 2$ and $q \geq 2$.

\subsubsection*{Size of the potential singular set}
Here we discuss our result in the context of partial regularity, more specifically, the size of the singular set in time or space-time. By singular times we mean the union of times at which the solution is not locally smooth, while singular points in space-time refer to points $(t,x)$ where the solution is not locally bounded (in the sense of $\text{ess}\sup$).

By the classical results of Leray, in 3D the Hausdorff dimension of possible singular times of a Leray-Hopf solution is bounded by $1/2$\footnote{This interpretation was made explicit in \cite{MR0452123}.}.  A key step in understanding the (possible) singular set of weak solutions was made by Scheffer \cite{MR454426,MR510154,MR573611} where the notion of suitable weak solutions was introduced. It was proved in~\cite{MR573611} that the singular sets of these suitable weak solutions have finite  $\frac{5}{3}$-dimensional Hausdorff measure in space-time. The theory of partial regularity culminated with the work~\cite{MR673830} by Caffarelli, Kohn, and Nirenberg where they show that $ \mathcal{P}^1(S)=0$, i.e., the $1$-dimensional parabolic Hausdorff measure of the singular set in space-time is zero. Note that these partial regularity results only provide upper bounds on the potential singular sets.

While convincing evidence~\cite{MR814542,MR895215,MR4066585} suggests that the upper bound of $1$-dimensional parabolic singularities in 3D is likely to be sharp for suitable weak solutions, it was unknown whether there exists a weak solution with a nontrivial\footnote{Here by nontrivial we mean that the singular set is not empty or full since smooth solutions have no singularity while the singular set of the solutions in \cite{MR3898708} is the whole space-time domain.} singular set until the work \cite{1809.00600} where the authors constructed wild solutions with a nonempty set of singular times with a dimension strictly less than $1 $. As in \cite{1809.00600}, solutions constructed here are not Leray-Hopf; however, they constitute the first example of 3D weak solutions that surpass the $1/2$ upper bound  (with a nonempty singular set). We remark that  the recent work~ \cite{2102.03244} proved that finite energy weak solutions with intervals of regularity are not typical.

In dimension $ d \geq 4$, the existence of partially regular (in space-time or in time) weak solutions becomes highly nontrivial. In fact, Leray's structure theorem only holds up to $d=4$ and the local energy inequality, a key ingredient in the partial regularity theory, remains absent in $d \geq 4$~\cite[Remark 1.1]{MR2318865}. Despite such a difficulty, the existence of partially regular weak solutions in space-time was established in 4D~\cite{MR510154} by Scheffer and also a recent result~\cite{2008.05802} by Wu. In dimension $d \geq 5$ the existence of partially regular weak solutions( in space-time or in time) was unknown to our knowledge and Theorem \ref{thm:main_thm_2} appears to be the first example of weak solutions with partial regularity in time in dimension $d \geq 5$.

Concerning the partial regularity in space-time,  the singular set of our solutions is the whole spatial domain at each singular time, as with all the other constructions exploiting a convex integration scheme. It might be possible to construct wild solutions that enjoy a certain space-time partial regularity by a space-time variant of the concentration procedure used here.  

\subsubsection*{Anomalous dissipation of the Euler equations}

A recent milestone in incompressible fluid dynamics is the resolution of the Onsager conjecture \cite{MR36116} which states that $\frac{1}{3}$-H\"older is the critical threshold for energy conservation for the 3D Euler equations. While the positive direction was settled in the 90s in \cite{MR1298949} following the first attempt by \cite{MR1302409} and then later refined in \cite{MR1734632,MR2422377}, the negative part was significantly harder and the regularity of counterexamples~\cite{MR1231007,MR1476315} was far below the threshold. Advances in the  negative direction really took off with the modern convex integration approach starting with the seminal paper of De Lellis and Székelyhidi Jr.~\cite{MR2600877}. The approach of using convex integration was refined and improved in a series of works~\cite{MR3090182,MR3254331,MR3374958,MR3530360}. Building upon these works, the threshold $C_t C^{1/3 - }$ was finally reached by Isett~\cite{MR3866888}, see also~\cite{1701.08678}. We remark that in the scale of $L^2$ Sobolev space, recently the authors in \cite{2101.09278} were able to show anomalous dissipation in $C_t H^{1/2-}$.

So far, constructed anomalous weak solutions have a limited regularity  on the whole time axis, namely, the H\"older regularity in space is always below $\frac{1}{3} $. The works~\cite{MR1298949,MR2422377} suggest that insisting on the exact ``$\frac{1}{3}$-H\"older regularity'' in space leads  to $L^3$ being the right critical scale in time for the energy conservation. This exact ``$\frac{1}{3}$-H\"older regularity''  of anomalous dissipation seems to be out of reach for the previous Euler schemes, an issue that has been investigated recently by Isett~\cite{1706.01549}.

Even though our inviscid solutions have a worse global-in-time regularity, they are smooth solutions on a ``large'' portion of the time axis,  and hence the kinetic energy is conserved locally in time. The mechanisms of the failure of the energy conservation are completely different: fast spatial oscillations play a key role in the previous Euler examples, whereas a strong temporal concentration here causes the breakdown at small time scales. It would be very interesting to combine the previous Euler results  with the current paper to show that there exist ``wild solutions'' in $C_t C^{\frac{1}{3} - } $ or $L^{3-}_t C^{\frac{1}{3} }$ that are locally smooth in time away from a small singular set, cf. \cite{2102.06085}. 

\subsection{Notations}
 
For reader's convenience, we collect the notations used throughout the manuscript.
\begin{itemize}
\item $\TT^d  = \RR^d / \ZZ^d$ is the $d$-dimensional torus and is identified with $[0,1]^d$. For any function $f: \TT^d \to \RR$ we denote by $f(\sigma \cdot)$ the $ \sigma^{-1} \TT^d$-periodic function $f(\sigma x)$.  The space $C^\infty_0(\TT^d)$ is the set of periodic smooth functions with zero mean and  $ C^\infty_0(\TT^d, \RR^d )$ is the set of periodic smooth vector fields with zero mean.

\item The Lebesgue space is denoted by $L^p$. For any $f \in L^1(\TT^d) $, its spacial average is
$$
\fint_{\TT^d} f \,dx= \int_{\TT^d} f\,dx.
$$ 
  For any function $f:[0,T] \times \TT^d \to \RR $, denote by $\| f(t) \|_p $ the Lebesgue norm on $\TT^d$ (in space only) at a fixed time $t$. If the norm is taken in space-time, we use $\|f \|_{L^p_{t,x}} $.

  \item For any Banach space $X$, the Bochner space $L^p(0,T;X)$ is equipped with the norm
$$
\Big(  \int_{0}^T \| \cdot  \|_X^p \, dt \Big)^\frac{1}{p},
$$
and we often use the short notations $L^p_t X$ and $\| \cdot \|_{L^p_t X}$. In particular, when $X = L^q(\TT^d)$, we write $  L^p_t L^q = L^p(0,T;L^q(\TT^d))$ for simplicity.

 \item The tensor divergence $\D  A  = \p_j A_{ij}$ for any matrix-valued function $A: \TT^d  \to \RR^{d\times d}$ and the tensor product $f \otimes g  = f_i g_j$ for any two vectors $f,g \in \RR^d$. The notion $\nabla$ indicates full differentiation in space only, and space-time gradient is denoted by $ \nabla_{t,x}$.

\item We write $X \lesssim Y$ if there exists a constant $C>0$ independent of $X$ and $Y$ such that $X \leq C Y $. If the constant $C$ depends on quantities $a_1,a_2,\dots,a_n$ we will write $X \lesssim_{a_1,\dots,a_n}$ or $X \leq C_{a_1,\dots ,a_n} Y $.

\end{itemize}

\subsection{Organization of the paper}
The organization of the rest of the paper is as follows.

\begin{enumerate}
    \item The outline of construction is given in  Section \ref{sec:outline}, where main theorems will be proved assuming the main proposition of the paper, proposition \ref{prop:main}.

    \item The proof of the main proposition is the content of the rest of the paper:
    \begin{enumerate}
        \item We concentrate the stress error to many small sub-intervals in Section~\ref{sec:proof_step_1};
        \item We design a velocity perturbation using convex integration to obtain a new solution pair $(u_1 , R_1)$ in Section \ref{sec:proof_step_2_convex_integration};
        \item Finally we estimate the perturbation along with the new stress error to conclude the proof in Section \ref{sec:proof_step_2}.
    \end{enumerate}

\item 
Appendix \ref{sec:append_weak} includes a proof of Theorem \ref{thm:FJR_uniqueness}.   Appendix \ref{sec:append_tech} contains some technical tools used in the paper, namely an improved H\"older's inequality and antidivergence operators on $\TT^d$.   
\end{enumerate}

%%%%%%%%%%%%%%%%%%%%%%%%%%%%%%%%%%%%%%%%%%%%%%%%%%%%%%%%%%%%%%%%%%%%%%%%%%%%%%%%%%%%%%%%%%%%%%%%%%%%%%%%%%%%%%%%%%%%%%%%%%%%%%
\section{Outline of the proof}\label{sec:outline}
%%%%%%%%%%%%%%%%%%%%%%%%%%%%%%%%%%%%%%%%%%%%%%%%%%%%%%%%%%%%%%%%%%%%%%%%%%%%%%%%%%%%%%%%%%%%%%%%%%%%%%%%%%%%%%%%%%%%%%%%%%%%%%

The proof of Theorem \ref{thm:main_thm_2} consists of an iterative scheme, which is achieved by repeatedly applying the main proposition of this paper, Proposition \ref{prop:main} to obtain a sequence of solutions $(u_n, R_n)$ to \eqref{eq:NSR}. 
The proof mainly consists of  three goals:
\begin{enumerate}
    \item The convergence of $u_n \to u$ in  $L^2_{t,x}$ and $R_n \to 0$ in $L^1_{t,x}$ so that $u$ is a weak solution of \eqref{eq:NSE}.
    \item Ensuring the final solution verifies $u \in L^p_t L^\infty \cap L^1_t W^{1,q}$ for $p<2$ and $q < \infty$.
    \item Achieving a small dimension of the singular set of $u$ in time.
\end{enumerate}

To this end, we employ a two-step approach:
\begin{itemize}
\item Step 1:  $(u_n, R_n) \xrightarrow{  \text{concentrating the stress error}} (\overline{u}_n , \overline{R}_n) $
\item Step 2: $ (\overline{u}_n , \overline{R}_n) \xrightarrow{\text{space-time convex integration} }(  u_{n+1} ,  {R}_{n+1} )  $
\end{itemize}
where the first step is mainly for achieving a small singular set in time and the second step is to ensure the convergence of $R_n$.

\subsection{The Navier-Stokes-Reynolds system}
Let us first introduce the approximate equations of \eqref{eq:NSE} for our approximate solutions. These approximate solutions solve the so-called Navier-Stokes-Reynolds systems
\begin{equation}\label{eq:NSR}
\begin{cases}
\p_t u - \Delta u  + \D(u \otimes u) + \nabla p = \D R &\\
\D u =0
\end{cases}
\end{equation}
where $R: [0,T]\times \TT^d \to \mathcal{S}^{d \times d}_0$ is a traceless symmetric matrix called Reynolds stress.

Since the associated pressure $p$ can be uniquely determined by the elliptic equation:
$$
\Delta p = \D \D R - \D\D( u \otimes u) =   \p_i \p_j ( R_{ij} - u_i u_j) 
$$
together with the usual zero spatial mean condition $\fint_{\TT^d} p \,dx = 0$, throughout the paper, we denote the solution of \eqref{eq:NSR} by $(u,R)$. 

This system \eqref{eq:NSR} arises naturally when studying weak solutions of the Navier-Stokes equations. The Reynolds stress $R$ emerges as the noncommutativity between average ensembles and the quadratic nonlinearity.

In the inviscid case, we can just drop the Laplacian term in \eqref{eq:NSR} and the system becomes the so-called Euler-Reynolds equations, which was widely used in constructing non-conserving weak solutions of the Euler equations in the context of Onsager's conjecture~\cite{MR3090182,MR3254331,MR3374958,MR3530360,MR3866888,1701.08678}. 

Since the Laplacian plays no role in our construction and is treated as a source of errors, in what follows we simply use \eqref{eq:NSR} to prove all the main theorems for the viscous case, and the results for the inviscid case can be obtained by dropping the Laplacian in \eqref{eq:NSR}.

\subsection{Concentrating the stress error}
As stated in the introduction, we proceed with two steps to prove this main proposition. The first step is a procedure that concentrates the stress error into many smaller sub-intervals.

Given $(u_{n-1},R_{n-1})$, we divide the time interval $[0,T]$ into smaller sub-intervals $J_i$ of length $\tau^{\ep}>0$, where $\tau>0$ will be chosen to be very small depending on $(u_{n-1},R_{n-1})$. So the total number of sub-intervals is $\sim \tau^{-\ep}$.

On each sub-interval $J_i$, we solve a generalized Navier-Stokes equations linearized around $(u_{n-1},R_{n-1})$ to obtain a corrector $v_i$ on $J_i$. More precisely, $v_i: J_i \times \TT^d \to \RR^d$ solves
\begin{equation}\label{eq:concentrate_stress}
\begin{cases}
\p_t v_{i} - \Delta v_{i } + \D( v_{i} \otimes v_{i}) + \D(v_{i} \otimes u  ) + \D( u \otimes v_{i}   ) + \nabla q_i=  -\D R &\\
\D v_i = 0 &\\
v_i(t_i) = 0
\end{cases}
\end{equation}
where $(u,R) = (u_{n-1} , R_{n-1})$, so that $u_{n-1} + v_i$ is an exact solution (of the Navier-Stokes equations) on $J_i$. The solvability of \eqref{eq:concentrate_stress} and smoothness of $v_i$ on $J_i$ are guaranteed by taking $\tau>0$ sufficiently small.

To concentrate the error and obtain a solution on $[0,T]$, we apply a sharp cutoff $\chi_i$ to the corrector $v_i$ and obtain the glued solution $ \overline{u}_{n-1}$ defined by
$$
\overline{u}_{n-1} := u_{n-1} +\sum_{i} \chi_i v_i.
$$
Specifically, each $\chi_i$ equals $1$ on a majority of the sub-interval $J_i$, but $ \chi_i =0 $ near endpoints of each $J_i$ of scale $\sim \tau$. Since $\ep\ll 1$, the cutoff $\chi_i$ is very sharp when comparing to the length of the sub-interval $J_i$. These sharp cutoffs $\chi_i$ ensure that $\overline{u}_{n-1} $ is well-defined and smooth on $[0,T]$.

On one hand, due to the sharp cutoff $ \chi_i$, the stress error $\overline{R}_{n-1}$ associated with $\overline{u}_{n-1}$ will only be supported near the endpoints of $J_i$ of time scale $\tau$.  In other words,  the temporal support of $\overline{R}_{n-1}$ can be covered by $\sim \tau^{-\ep}$ many intervals of size $\sim \tau$, from which one can already see a small dimension of the singular set of the final solution.

On the other hand,  the corrector $v_i$ is very small, say in $L^\infty_t H^d$, since it starts with initial data $0$ and we can choose time scale $\tau^{\ep} = |J_i|$ to be sufficiently small. More importantly, we can show that the new stress error $\overline{R}_{n-1}$ associated with $\overline{u}_{n-1}$ satisfies the estimate
$$
\| \overline{R}_{n-1} \|_{L^1_t L^r} \lesssim_{\ep,r} \|  {R}_{n-1}  \|_{L^1_t L^r} \quad \text{for all $1 < r< \infty$},
$$
with an implicit constant independent of the time scale $\tau>0$. In other words, concentrating the stress error $R_{n-1}$ to $ \overline{R}_{n-1}$ cost a loss of a constant multiple when measuring in $L^1$ norm in time.

\subsection{Space-time convex integration}

The next step is to use a convex integration technique to reduce the size of $ \overline{R}_{n-1}$ by adding further a perturbation $w_n$ to $\overline{u}_{n-1}$ to obtain a new solution $(u_n, R_n)$ of \eqref{eq:NSR}. The perturbation $w_n$ and the new stress $R_n$ satisfies the equation
$$
\D R_{n}=  \D \overline{R}_{n-1} +\D(w_{n}  \otimes w_{n} ) +  \p_t w_n  -\Delta w_{n} + \D(\overline{u}_{n -1 }  \otimes w_{n} ) + \D(w_{n}  \otimes \overline{u}_{n -1} )  + \nabla P_n,
$$
for a suitable pressure $P_n$. The heuristic is that the high-high to low cascade in space-time  of $w_{n}  \otimes w_{n}$ can balance the old stress error $\overline{R}_{n-1}$ in the sense that
\begin{equation}\label{eq:outline_cascade}
  \D(\overline{R}_{n-1} + w_{n}  \otimes w_{n} )= \text{High Spacial Freq. Term} + \text{High Temporal Freq. Term}+  \text{Lower Order Terms},   
\end{equation}
where the ``High Temporal Freq. Term'' above will further be balanced by a part of $\p_t w_n$, as in~\cite{MR3898708,1809.00600} and~\cite{2004.09538}. However, one of the fundamental differences to ~\cite{MR3898708,1809.00600} is that this additional ``convex integration in time'' requires no additional constraint of oscillation and concentration and is basically free, which is crucial to obtain the sharp regularity $L^p_t L^\infty \cap L^1_t W^{1,q}$.

In particular, executing the scheme of \cite{2004.09538} requires two crucial ingredients:
\begin{enumerate}
    \item Suitable stationary flows as the spatial building blocks that can achieve some level of spatial concentration.
    \item The use of intermittent temporal functions to oscillate the spatial building blocks in time.
\end{enumerate}

Once $(1)$ is available, it is relatively straightforward to implement $(2)$. On the technical side, we use the stationary Mikado flows introduced in \cite{MR3614753} as the spatial building blocks. These are periodic pipe flows that can be arranged to be supported on periodic cylinders with a small radius. In other words, Mikado flows can achieve a $d-1$-dimensional concentration on $\TT^d$, which is more than enough in view of $(1)$. It is worth noting that in the framework of \cite{MR3898708}, stationary Mikado flows are not sufficiently intermittent to be used for the Navier-Stokes equations in dimension $d \leq 3$, cf. \cite{MR3951691}.

The space-time cascade~\eqref{eq:outline_cascade} imposed a relation between the perturbation $w_n$ and the stress error $\overline{R}_{n-1}$ as
\begin{equation}\label{eq:outline_w_n}
  \|w_n \|_{L^2_{t,x}} \sim \| \overline{R}_{n-1} \|_{L^1_{t,x}}.
\end{equation}
The relation \eqref{eq:outline_w_n} will imply the convergence in $L^2_{t,x}$ of the approximate solutions $u_{n}$ as long as one can successfully reduce the size of the stress error
\begin{equation}\label{eq:outline_R_n}
 \|  {R}_{n } \|_{L^1_{t,x}} \ll \| \overline{R}_{n-1} \|_{L^1_{t,x}}   .
\end{equation}

In particular, special attention will be paid to estimating the temporal derivative part of the new stress error 
\begin{equation}\label{eq:outline_1}
  \D R_{\Tem } =     \p_t w_n,
\end{equation}
and achieving the regularity of the perturbation
\begin{equation}\label{eq:outline_2}
\| w_n\|_{L^p_t L^\infty } + \| w_n\|_{L^1_t W^{1,q} } 
 \ll 1 .
\end{equation}

These two constraints \eqref{eq:outline_1} and \eqref{eq:outline_2} require a very delicate choice of parameters when designing the perturbation $w_n$. On one hand, \eqref{eq:outline_1} implies the temporal frequency can not be too large, relative to the spatial frequency, otherwise the time derivative will dominate. On the other hand, \eqref{eq:outline_2} requires a large temporal frequency so that temporal concentration can offset the loss caused by going from $L^2$ to $L^\infty$ or  $W^{1,q}$ in space in relation to \eqref{eq:outline_w_n}. 

Nevertheless, it turns out that the scheme in \cite{2004.09538} is flexible enough to accommodate \eqref{eq:outline_1} and \eqref{eq:outline_2}. We could somehow explain why this is possible. Roughly speaking, the method in \cite{2004.09538} is $L^2_{t,x}$-critical.  While it is difficult for $w_n$ to go above the $L^2_{t,x}$ regularity, we trade $L^2$ for $L^p$ (resp. $L^1$) in time to obtain an improvement of $L^2$ to $L^\infty $ (resp. $W^{1,q}$) in space.  We provide a scaling analysis below.

 \subsection{Oscillation and concentration}\label{subsec:osc_con}

We do this computation in general dimension $d \geq 2$ and $D \in [0,d]$ denotes the spatial intermittency, cf. \cite{MR0495674,MR1428905}.

We start with a velocity perturbation in $L^2_{t,x}$ with a certain decay given by the previous stress error,
$$
\|w_n \|_{L^2_{t,x}} \to 0 \quad \text{as $n \to \infty$}.
$$
Denote the spatial frequency by $  \l$ and the temporal frequency by $ \kappa$, namely
$$
\|\p_t^m \nabla^s w_n \|_{L^2_{t,x}} \lesssim \kappa^m  \l^s.
$$
The intermittency parameter $D\in [0,d]$ in space dictates the concentration level of $w_n$ and the scaling law
\begin{equation}\label{eq:outline_3}
  \| w_n (t) \|_{L^q} \lesssim \| w_n (t)\|_{L^2}  \l^{(d-D)( \frac{1 }{2} - \frac{1}{q})} \quad \text{for all $1 \leq q \leq \infty$}.
\end{equation}
As for the temporal scaling, we assume for simplicity a full dimensional concentration in time:
\begin{equation}\label{eq:outline_4}
\| w_n \|_{L^p_t L^q} \lesssim \kappa^{\frac{1}{2} - \frac{1}{p}}\| w_n \|_{L^2_t L^q}  \sim \kappa^{\frac{1}{2} - \frac{1}{p}}  \l^{(d-D)( \frac{1 }{2} - \frac{1}{q})} \quad \text{for all $1 \leq p,q \leq \infty$}.
\end{equation}

With such scaling laws, we effectively assume a negligible amount of temporal oscillation and the goal then boils down to finding a working choice of $D$ in terms of the given parameters $d,p,q$. In other words, we need to find a balance between spatial oscillation and spatial concentration.

By the scaling relations \eqref{eq:outline_3} and \eqref{eq:outline_4}, the stress error contributed by the time derivative \eqref{eq:outline_1} satisfies
\begin{equation}\label{eq:outline_5}
 \|  \D^{-1}(     \p_t w_n) \|_{L^1_{t,x}} \lesssim  \kappa^{ \frac{1}{2}}  \l^{-1}  \l^{- \frac{d-D}{2}},
\end{equation}
where we assume formally an inverse divergence $ \D^{-1}$ gains one full derivative in space.
The regularity condition \eqref{eq:outline_2} becomes
\begin{equation}\label{eq:outline_6}
\| w_n\|_{L^p_t L^\infty } \sim \kappa^{\frac{1}{2} -\frac{1}{p}}  \l^{ \frac{d-D}{2}} \ll 1,
\end{equation}
and
\begin{equation}\label{eq:outline_7}
 \| w_n\|_{L^1_t W^{1,q} } \sim \kappa^{-\frac{1}{2}}  \l^{1 + \frac{d-D}{2} - \frac{d-D}{q}} \ll 1.
\end{equation}

In particular, \eqref{eq:outline_5} and \eqref{eq:outline_7} imply that
\begin{equation}\label{eq:outline_8}
 \l^{1 + \frac{d-D}{2} - \frac{d-D}{q}} \ll \kappa^{\frac{1}{2}}  \ll  \l^{1 + \frac{d-D}{2}},
\end{equation}
which always has some room since $q<\infty$. Then all we need to is to choose $D$ to ensure \eqref{eq:outline_6} and \eqref{eq:outline_8}. One can already see that $D$ should be very close to $d$, which means we need much more spatial oscillation than spatial concentration. Indeed, solutions to \eqref{eq:outline_6} and \eqref{eq:outline_8} do exist and we refers to Section \ref{subsection:choice_of_parameters} for the exact choice used.

%%%%%%%%%%%%%%%%%%%%%%%%%%%%%%%%%%%%%%%%%%
\subsection{The main iteration proposition}
%%%%%%%%%%%%%%%%%%%%%%%%%%%%%%%%%%%%%%%%%%

We are ready to introduce the main iteration proposition of the paper that materializes the above discussion. To simplify the presentation, let us introduce the notion of well-preparedness of solutions to \eqref{eq:NSR}, which encodes the small Hausdorff dimension of the singular set in time. Throughout the paper, we take $T=1$ and assume $0<\ep<1$ without loss of generality.
\begin{definition}
We say a smooth solution  $(u, R)$ of \eqref{eq:NSR} on $[0,1]$ is well-prepared if there exist a set $I$ and a length scale  $ \tau  >0$ such that $I$ is a union of at most $ \tau ^{-\ep}$ many closed intervals of length $5\tau $ and
$$
R(t,x) =0 \quad \text{if} \quad \dist(t, I^c) \leq \tau .
$$.
\end{definition}
With this definition, to ensure the solution $u$ has intervals of regularity with a small residue set of Hausdorff dimension $\lesssim \ep$, it suffices to construct approximate solutions $(u_n ,R_n)$ that are well-prepared for some $I_n$ and $\tau_n$ such that
$$
I_{n} \subset I_{n-1} \quad \text{and}\quad \tau_n \to 0.
$$

The main proposition of this paper states as follows.

\begin{proposition}[Main iteration]\label{prop:main}
For any $\ep>0$ and $p < 2$, there exists a universal constant $M=M(\ep,p)>0$ and $r>1$ depending only on $p$ and $q$ such that the following holds.

Let $\delta>0$ and  $(u,R)$ be a well-prepared smooth solution of \eqref{eq:NSR} for some set $\widetilde{I}$ and a length scale $ \widetilde{\tau }>0$. Then there exists another well-prepared smooth solution $(u_1,R_{1})$  of \eqref{eq:NSR} for some set $I \subset \widetilde{I}$ with  $ 0,1 \notin I $ and $\tau  < \widetilde{\tau}/2$ such that 
$$
\|R_{1}  \|_{L^1(0,1; L^r( \TT^d  ))} \leq \delta .
$$
Moreover, the velocity perturbation  $w : = u_{1}  - u $ satisfies
\begin{enumerate}
\item  The temporal support of $w$ is contained in $I$, i.e.
$$
\Supp w \subset I \times \TT^d;
$$

\item The $L^2_{t,x} $ estimate,
$$
 \|w \|_{L^2([0,1] \times \TT^d )} \leq M  \| R \|_{L^1([0,1] \times \TT^d )} ;
$$

\item The $L^{p}_t L^\infty \cap L^{1}_t W^{1,q}   $  estimate,
$$
\|w \|_{L^{p }(0,1;   L^\infty (\TT^d))} + \|w \|_{L^{1 }(0,1;   W^{1,q}  (\TT^d))} \leq    \delta .
$$

\end{enumerate} 

\end{proposition}

\begin{remark}
We list a few comments concerning the main proposition.
\begin{enumerate}
\item It will be clear that the construction adapts no change for the Euler equations, except dropping the Laplacian. 

\item The parameter $r>1$ is used to ensure the $L^r$ boundedness of the Calder\'on-Zygmund singular integral and is very close to $1$.
    
\item Due to the local well-preparedness, on large portions of the time axis, the solutions are exact solutions of the Navier-Stokes equations (or the Euler equations) and we do not touch them in the future.

\end{enumerate}
\end{remark}

We will break down Proposition \ref{prop:main} into two separate propositions, whose proof will be the context of Section \ref{sec:proof_step_1} and respectively Section \ref{sec:proof_step_2_convex_integration} and Section \ref{sec:proof_step_2}.

%%%%%%%%%%%%%%%%%%%%%%%%%%%%%%%%%%%%%%%%%%%%%%%%%%%%%%%%%%%%%%%%%%%%%%%%%%%
\subsection{Proof of main theorems}
%%%%%%%%%%%%%%%%%%%%%%%%%%%%%%%%%%%%%%%%%%%%%%%%%%%%%%%%%%%%%%%%%%%%%%%%%%%
We first deduce Theorem \ref{thm:main_short} and Theorem \ref{thm:main_sharp_short} from Theorem \ref{thm:main_thm_2}.

\begin{proof}[Proof of Theorem \ref{thm:main_short}]
Given initial data $u_0 \in L^2(\TT^2) $, we let $v$ be the Leray-Hopf weak solution on $[0,1]$ with initial data $v(0) = u_0$.

Since $ {v}$ is Leray-Hopf on $[0, 1]$,   $ {v}$ is in fact smooth on $(0, 1]$. Let $\widetilde{v}:[1/2,1] \times \TT^2 \to \RR^2  $ be a smooth divergence-free vector field that coincide with $v$ on $[1/2, 3/4]$ but
\begin{equation}\label{eq:proof_thm1.5_1}
\| \widetilde{v} - v \|_{L^p(1/2,1; L^\infty(\TT^d)} \geq 1 .
\end{equation}
Then we can apply Theorem \ref{thm:main_thm_2} for the vector field $ \widetilde{v} $  and some $0<\ep<1$ to obtain a weak solution $\tilde{u} $ on $[ 1/2 ,1] \times \TT^2$. The conclusion follows once we define the weak solution $u:[0,1]\times \TT^2 \to \RR^2$ by
\begin{equation*}
u = 
\begin{cases}
{v} & \text{if $t   \in [0, 1/2 ]$}\\
\widetilde{u} & \text{if $t \in [1/2,1]$}.
\end{cases}
\end{equation*}
 
Indeed, the new glued solution $u$ is still a weak solution of \eqref{eq:NSE} due to the fact that $\widetilde{u}$ and $\widetilde{v}$ coincide on $[1/2, 1/2+\delta]$ for some $\delta>0$. Moreover, by \eqref{eq:proof_thm1.5_1}
$$
\|u -v \|_{L^p( 1/2,1 ; L^\infty} \geq \|    \widetilde{v}  - v \|_{L^p( 1/2,1 ; L^\infty}- \| \widetilde{u}  -\widetilde{v}  \|_{L^p( 1/2,1 ; L^\infty}  \geq 1-\ep >0,
$$
which implies $u \neq v$. The $L^2$ continuity at $t=0$ follows from the fact that $ u_{[0,1/2]}\equiv  v_{[0,1/2]}$ and $v$ is the Leray-Hopf solution.

\end{proof}

\begin{proof}[Proof of Theorem \ref{thm:main_sharp_short}]
To prove both points, it suffices to show that given a weak solution  $v \in L^p(0,1; L^\infty(\TT^d)$  of \eqref{eq:NSE} with at least one interval of regularity, there exists a non-Leray-Hopf weak solution $u \in L^p(0,1; L^\infty(\TT^d)$ having intervals of regularity and with the same initial data.

Let $[a,b]$ be an interval of regularity of the weak solution $v$. We first choose a smooth, divergence-free vector field $\widetilde{v}:[a,1]\times \TT^d \to \RR^d $ such that 
$$ 
\widetilde{v}|_{ [a,b] \times \TT^d} \equiv v|_{ [a,b] \times \TT^d},
$$ 
but
\begin{equation}\label{eq:proof_thm1.6_1}
\|\widetilde{v} - v \|_{L^p(  a,1 ; L^\infty(\TT^d) )} \geq 1  \quad \text{and} \quad  \|\widetilde{v}   \|_{L^p(  a,1 ; L^2(\TT^d) )} \geq 1+ \| v(0)\|_{L^2(\TT^d) }, 
\end{equation}
As in the proof of Theorem \ref{thm:main_short}, we apply Theorem \ref{thm:main_thm_2}   to $ \widetilde{v}$ and $\ep<1$ to obtain a weak solution $\widetilde{u}   \in L^p(  a,1 ; L^\infty(\TT^d)) $ such that
\begin{equation}\label{eq:proof_thm1.6_2}
\|\widetilde{u} -\widetilde{v} \|_{L^p(  a,1 ; L^\infty(\TT^d) )} \leq \ep  .
\end{equation}

We then define a new solution $u: [0,1] \times \TT^d \to \RR^d$ by
\begin{equation*}
u = 
\begin{cases}
v & \text{if $t \in [0,a]$}\\
\widetilde{u}  & \text{if $t \in [a, 1]$}.
\end{cases}
\end{equation*}
The glued solution $u \in L^p_t L^\infty$ is still a weak solution due the smoothness of both $v$ and $\widetilde{u} $ near $t =a$. Next, $u$ and $v$ are different  since
$$
\| u -v \|_{L^p_t L^\infty } = \|\widetilde{u}  -v \|_{L^p ( a,1  ;L^\infty )} \geq \|\widetilde{v}  - v \|_{L^p ( a,1  ;L^\infty )}  - \|\widetilde{u}  - \widetilde{v} \|_{L^p ( a,1  ;L^\infty )} \geq 1 -\ep ,
$$
where we have used \eqref{eq:proof_thm1.6_1} and \eqref{eq:proof_thm1.6_2}.
Finally, $u$ can not be a Leray-Hopf solution since 
$$
\| u \|_{L^p_t L^2} \geq \|\widetilde{u}  \|_{L^p ( a,1  ;L^2 )} \geq \|\widetilde{v}  \|_{L^p ( a,1  ;L^2 )} -\|\widetilde{u}  - \widetilde{v} \|_{L^p ( a,1  ;L^\infty )}  > \|v(0) \|_2
$$
and Leray-Hopf solutions must have a non-increasing $L^2$ norm.
\end{proof}

Next, we prove Theorem \ref{thm:main_euler_onsager} in the case of the Euler equations. The proof is identical for the Navier-Stokes equations since Proposition \ref{prop:main} holds for both equations.

\begin{proof}[Proof of Theorem \ref{thm:main_euler_onsager}]
We first choose an infinite sequence of smooth divergence-free vector fields $v_n \in C^\infty_0([0,1] \times \TT^d) $ such that:
\begin{enumerate}
    \item On $[0, \frac{1}{2}]$, every $v_n$ coincides and is equal to an exact solution of the Euler equations;
    \item Each $v_n$ satisfies
    $$
    \| v_n \|_{L^p( 0,1 ; L^\infty)}= 2^n.
    $$
\end{enumerate}

We now apply Theorem \ref{thm:main_euler} for the Euler equations with the vector fields $v_n$ and some $\ep<  \frac{1}{2}$ to obtain infinite many weak solution $u_n \in L^p_t L^2 \cap L^1_t W^{1,q}$ of the Euler equations.

On one hand, since each $v_n$ we used agrees and solves the Euler equations on $[0,1/2]$, by Theorem \ref{thm:main_thm_2}, for any $n \geq 1$, there exists $\tau_n>0$ such that $u_n$ coincides with $v_n $ on $[0,\tau_n]$\footnote{In fact, $\tau_n$ can be taken to be $1/2$ if we use Proposition \ref{prop:main} instead of Theorem \ref{thm:main_thm_2}.}. So every weak solution $u_n$ has the same initial data.

On the other hand, these weak solutions $u_n$ are different since for any $n > m$,
$$
\|u_n - u_{m}\|_{L^p_t L^\infty} \geq \|u_n - v_{n}\|_{L^p_t L^\infty}  -\|u_m - v_{m}\|_{L^p_t L^\infty}  \geq  2^{n}-1,
$$
where we have used that $ \ep<  \frac{1}{2}$.

It remains to show the regularity $u_n \in  L^{3/2 -\ep}_t C^{1/3} \cap L^1_t C^{1-\ep}  $. This can be done by standard interpolations. Since for any $\ep>0$ there exists $q<\infty$ such that the embedding
$$
W^{1,q}(\TT^d) \hookrightarrow   C^{1-\ep}(\TT^d) 
$$
holds, we get $L^1_t C^{1-\ep} $. The bound $ u \in L^{3/2 -\ep}_t C^{1/3}$ can then be obtained by interpolating $L^p_t L^\infty$ with $L^1_t C^{1-\ep{}}$ (with different $\ep$).
\end{proof}

Finally, we prove Theorem \ref{thm:main_thm_2} and Theorem \ref{thm:main_euler} assuming Proposition \ref{prop:main}.
\begin{proof}[Proof of Theorem \ref{thm:main_thm_2} and Theorem \ref{thm:main_euler}]

Let $u_0 = v $ and
$$
R_0 = \mathcal{R}\Big(  \p_t u_0 -\Delta u_0 + \D(u_0 \otimes u_0 ) \Big)
$$
where $\mathcal{R}$ is an inverse divergence operator on $\TT^d$ defined in Appendix \ref{sec:append_tech}.

Since the given vector field $v$ has zero spatial mean for each $t\in [0,1]$, $(u_0 ,R_0)$ solves \eqref{eq:NSR} trivially and is well-prepared for  $I=[0,1]$ and $\tau=1$. We construct a sequence of solution $(u_n ,R_n)$ for $n \in \NN$ as follows.

Given $(u_{n-1} ,R_{n-1})$, we apply Proposition~\ref{prop:main} with the parameter
$$
  \delta_n := 2^{-n} \min\left\{  \|R_{n-1} \|_{L^1 L^r} , \ep \right\}
$$
to obtain $(u_{n} ,R_{n} )$. Denote the perturbations by $w_{n } : =u_{n } - u_{n-1}$ for $n \geq 1 $.

As a result, we have
$$
\| R_n \|_{L^1_{t,x}} \leq \| R_n \|_{L^1_t L^r} \leq \delta_{n},
$$
and
\[
\|w_{n} \|_{L^p_t L^\infty} + \|w_{n} \|_{L^1_t W^{1,q}} \leq     \delta_{n},
\]
for any $n \geq 1 $. Also,
\begin{align*}
\|w_{n} \|_{L^2_{t,x}} \leq M    \| R_{n-1} \|_{L^1_{t,x} } \leq M    \delta_{n-1},
\end{align*}
for any $n \geq 2$.
 
Since $u_n$ is Cauchy in $L^2_{t,x} \cap L^p_t L^\infty 
\cap L^1_t W^{1,q}$, there exists  $u \in  L^2_{t,x} \cap L^p_t L^\infty 
\cap L^1_t W^{1,q} $ such that
$$
u_n \to u  \quad \text{in} \quad L^2_{t,x}  \cap L^p_t L^\infty \cap L^1_t W^{1,q}.
$$
To show that $u$ is a weak solution of \eqref{eq:NSE}, we need to verify the $a.e.$ in time divergence-free condition and the weak formulation. As we will show below, the Lebesgue measure of the singular set in time is zero, so $u$  is $a.e.$ in time divergence-free by construction.

Now we show that the weak formulation holds. Indeed, take any $\varphi \in \mathcal{D}_T$, then using the weak formulation of \eqref{eq:NSR} or integrating by parts we have
$$
\int_{\TT^d} u_n(0,x) \varphi(0,x ) \, dx = -   \int_{[0,1]\times \TT^d} (u_n\cdot \Delta \varphi + u_n\otimes u_n : \nabla \varphi + u_n \cdot \partial_t \varphi)  \, dx dt  - \int_{[0,1]\times \TT^d} R_n : \nabla \varphi  \, dx dt .
$$
Since $u_n \to u $ in $L^2_{t,x}$, $R_n \to 0$ in $L^1_{t,x}$, and $u_n(0,x) \equiv v(0,x)$ for all $n \geq 1$, it follows that all terms above converge to their natural limits and hence $ u$ is a weak solution of \eqref{eq:NSE}.

Moreover,
$$
\| u -v \|_{L^p_t L^\infty 
\cap L^1_t W^{1,q}} \leq \sum_{n\geq 1}\left[ \| w_n \|_{L^p_t L^\infty } +  \| w_n \|_{L^1_t W^{1,q} }\right] \leq \sum_{n\geq 1} 2^{-n} \ep  \leq \ep.
$$

Finally, we show the structure of the intervals of regularity of $u$.  Recall that each solution $(u_n , R_n)$ is well-prepared, so let us denote by $\mathcal{B}_n \subset [0,1] $ and $ \tau_n >0$   the set and length scale of the well-preparedness of $(u_n , R_n)$.

%%%%%%%%%%%%%%%%%%%%%%%%%%%%%%%%%%%%%%%%%%%%%%%%%%%%%%%%%%%%%%%
 
Let
\[
\mathcal{I} =   \bigcup_{n \geq 0} \mathcal{B}_n^c \setminus \{0,1\},
\]
where the complement is taken in $[0,1]$. Note that $\mathcal{B}_{n}^c \subset  (\Supp w_n)^c  $ for all $n \geq 1$, and $\mathcal{B}_n^c\setminus \{0,1\} $ are monotonically increasing open sets. Therefore $w_k(t)\equiv 0$ on $\mathcal{B}_{n}^c$ for all $k\geq n$, and hence $u(t) \equiv  u_n(t)$ on $\mathcal{B}_{n}^c$ for each $n$. Since $u_n$ is smooth, this proves that $u|_{\mathcal{I}\times \TT^d }\in C^\infty(\mathcal{I} \times \TT^d)$. By construction, each $u_n$ also agrees with each other for $0 \leq t \leq \tau_1 /2$, and hence $u$ agree with the unique smooth solution with initial data $v(0)$ near $t=0$.

Finally, since each $\mathcal{B}_{n}$ is a finite union of closed intervals, we can write $\mathcal{I} $ as a union of countably many open intervals:
$$
\mathcal{I} = \bigcup_{i\geq 0} (a_i, b_i).
$$
For the Hausdorff dimension bound of $[0,1] \setminus\mathcal{I}$,  we notice that
$$
(0,1) \setminus\mathcal{I} = \bigcap_{n \geq 0} \mathcal{B}_n = \limsup_n  \mathcal{B}_n .
$$
Since each $\mathcal{B}_n  $ is covered by at most $\tau_n^{-\ep}$ many balls of radius $5 \tau_n $, we have
$$
d_{\mathcal{H}} (\limsup_n  \mathcal{B}_n  ) \leq \ep .
$$

\end{proof}

%%%%%%%%%%%%%%%%%%%%%%%%%%%%%%%%%%%%%%%%%%%%%%%%%%%%%%%%%%%%%%%%%%%%%%%%%%%
%%%%%%%%%%%%%%%%%%%%%%%%%%%%%%%%%%%%%%%%%%%%%%%%%%%%%%%%%%%%%%%%%%%%%%%%%%%
\section{Concentrating the stress error}\label{sec:proof_step_1}
%%%%%%%%%%%%%%%%%%%%%%%%%%%%%%%%%%%%%%%%%%%%%%%%%%%%%%%%%%%%%%%%%%%%%%%%%%%%%%%%%%%%%%%%%%%%%%%%%%%%%%%%%%%%%%%%%%%%%%%%%%%%%%

The goal of this section is to prove Proposition \ref{prop:main_step_1} below. The idea is that given a solution $(u,R)$ of \eqref{eq:NSR}, we can add a small correction term to the existing solution $(u,R)$   so that  all of the stress error $R$ concentrates to a set $I$, the union of small intervals of length $\tau$, and thus obtain a new solution $( \overline{u}, \overline{R})$. The key is that the procedure
$$ 
( u  , R ) \rightarrow  ( \overline{u}  ,\overline{R}  )
$$ 
leaves the size of the stress error $R$ invariant $L^1$ in time, up to a cost of a constant multiple, namely
\begin{equation}\label{eq:R_concentration_loss}
\| \overline{R} \|_{L^1_t L^r} \leq C  \|  R  \|_{L^1_t L^r} \quad \text{for any $1 < r < \infty$},   
\end{equation}
where $C=C(r,\ep,d) $ is a universal constant that only depends on $d$ and $r,\ep$ in the well-preparedness.

\begin{proposition}\label{prop:main_step_1}
Let $ 0 < \ep <1$ and  $(u,R)$ be a well-prepared smooth solution of \eqref{eq:NSR} for some set and a length scale $\widetilde{I}$ and $ \widetilde{\tau} >0$. For any $ 1 < r <\infty$, there exists a universal constant $C=C(r,\ep,d)>0$ such that the following holds.

For any $\delta>0$, there exists another well-prepared smooth solution $(\overline{u},\overline{R})$  of \eqref{eq:NSR} some set $I \subset \widetilde{I}$ with $0 ,1 \not \in I$ and $\tau  <\widetilde{\tau}/2$  satisfying the following. 
\begin{enumerate}
\item The new stress error $\overline{R}$ satisfies
$$
\overline{R}(t,x) = 0  \quad \text{if }   \,\,\dist(t,I^c) \leq \frac{3\tau}{2} ,  
$$
and
$$
\|\overline{R}  \|_{L^1(0,1; L^r( \TT^d  ))}   \leq C \|R  \|_{L^1(0,1; L^r( \TT^d  ))}  
;
$$

\item  The velocity perturbation  $ \overline{w} : = \overline{u}  - u $ satisfies
$$
\Supp \overline{w}  \subset\widetilde{I} \times \TT^d,
$$
and
$$
\|\overline{w}  \|_{L^{\infty}(0,1;   H^d (\TT^d))} \leq    \delta.
$$

\end{enumerate}

\end{proposition}

Note the slightly stricter bound $\dist(t,I^c) \leq \frac{3\tau}{2} $  versus the definition of well-preparedness is to leave room for the future convex integration scheme in the next section.

\subsection{Subdividing the time interval}
We first introduce a subdivision of the time interval $[0,1]$. Then on each sub-interval $[t_i, t_{i+1}]$, we solve a generalized Navier-Stokes (or Euler in the inviscid case) equations and obtain a solution $v_i$ so that $u+v_i$ is a exact solution of the Navier-Stokes equations on $[t_i, t_{i+1}]$.

Let $\tau >0$ be a small length scale to be fixed in the end of this section and define for $ 0 \leq i \leq \lfloor \tau^{-\ep} \rfloor $ 
$$
t_i = i \tau^{\ep}.
$$
Without loss of generality, we assume $\tau^{-\ep}$ is always an integer so that the time interval $[0,1]$ is perfectly divided.

For $ 0 \leq i \leq   \tau^{-\ep} -1  $,  let $v_i : [t_i, t_{i+1}] \times \TT^d \to \RR^d$ and $q_i : [t_i, t_{i+1}] \times \TT^d \to \RR $ be the solution of the following generalized Navier-Stokes system
\begin{equation}\label{eq:concentrator}
\begin{cases}
\p_t v_{i} - \Delta v_{i } + \D( v_{i} \otimes v_{i}) + \D(v_{i} \otimes u  ) + \D( u \otimes v_{i}   ) + \nabla q_i=  -\D R &\\
\D v_i = 0 &\\
v_i(t_i) = 0.
\end{cases}
\end{equation}

Since the initial data for $v_i$ is zero and $u$ and $R$ are smooth on $[0,1]\times \TT^d$, thanks to the general local wellposedness theory of the Navier-Stoke equations (or the Euler equations in the inviscid case), for all sufficiently small $\tau>0$, we may solve equation \eqref{eq:concentrator}  on intervals $t \in [t_i ,t_{i+1}]$ to obtain a unique smooth solution $v_i$.

We shall focus on estimating each $v_i$ on the associated interval $[t_i, t_{i+1}]$. The solution $v_i$ serves as an ``accumulator'' of the stress error on $[t_i, t_{i+1}]$, and it will provide the major contribution to the new stress error $\overline{R}$ once we use a gluing procedure.

Recall that $\mathcal{R} : C^\infty(\TT^d ,\RR^d) \to C^\infty(\TT^d, \mathcal{S}^{d \times d }_0)$ is an inverse divergence operator on $\TT^d$ defined in Appendix \ref{sec:append_tech}. The below result quantifies the size of the corrector $v_i$ in relation to the time scale $\tau$ and the forcing $-\D R$.
\begin{proposition} \label{estimates_on_vi}
Let  $d\geq 2$ and  $(u,R)$ be a  smooth solution of \eqref{eq:NSR}. There exists a universal constant $C_r>0$ depending on $1< r <\infty$ and $d$ so that the following holds. 

For any $\delta>0$, if $\tau>0$ is sufficiently small, then the unique smooth solutions $v_i$ to \eqref{eq:concentrator} on $[t_i, t_{i+1}]$ satisfies
$$
\| v_i \|_{L^\infty (   t_i , t_{i+1} ;  H^d(\TT^d)) } \leq  \delta,
$$
and
$$
\| \mathcal{R} v_i \|_{L^\infty(   t_i , t_{i+1} ;  L^r(\TT^d))   } \leq C_r  \int_{ t_i}^{t_{i+1} } \| R (t )\|_{ L^r }  \, dt  + C_u \delta \tau^{\ep}  ,
$$
where $C_u$ is a sufficiently large constant depending on $u$ but not on $\delta$ or $\tau$.
\end{proposition}
\begin{proof}
The first estimate follows directly from the fact that $\|u\|_{H^d}$ is uniformly bounded on $[0,1]$ and standard energy bounds for $v_i$ by a continuity argument. Assuming $\delta>0$ is sufficiently small, we prove the second one as follows.

To reduce notations, we simply write $v$ for $v_i$ and denote $z : = \mathcal{R} v   $. Note that \eqref{eq:concentrator} preserves the zero-mean condition.  

Denote by $\mathbb{P}$ the Leray projection onto the divergence-free vector fields on $\TT^d$. Once the pressure is eliminated by projecting \eqref{eq:concentrator}, the evolution of $z$ is governed by
\begin{align*}
\p_t z -\Delta z = F,
\end{align*}
where
\begin{align*}
F =  - \mathcal{R} \mathbb{P}\D ( v\otimes v + u \otimes v + v\otimes u )  - \mathcal{R} \mathbb{P} \D  R .
\end{align*} 

Since $1 < r <\infty$ and $\mathcal{R}\mathbb{P}\D$ is a Calder\'on-Zygmund operator on $L^r(\TT^d)  $, a standard energy method yields
\begin{align*}
\|z(t) \|_{L^r(\TT^d)} \leq C_r \int_{t_i }^{t}\left(\| R (s)\|_r  + \|v\otimes v \|_r + \|u\otimes v \|_r + \|v\otimes u \|_r\right)\, ds \quad \text{for all $t \in [t_i , t_{i+1}]$,}
\end{align*}
where the constant $  C_r$ depends only on $r>1$ and dimension $d$. 
Using the obtained estimate on $v$ and the embedding $H^d(\TT^d) \subset L^\infty(\TT^d)$, we get
\begin{align*}
\|z \|_{L^\infty(   t_i , t_{i+1} ;  L^r(\TT^d))   }  & \leq C_r  \int_{t_i }^{t_{i+1}}\left(\| R (t)\|_r    + \|v  \|_{L^\infty(   t_i , t_{i+1} ;  H^d(\TT^d))   }^2  +  \| u \|_{L^\infty_{t,x }} \|v  \|_{L^\infty(   t_i , t_{i+1} ;  L^r(\TT^d))   }\right) \, dt 
\end{align*}
Since $ t_{t+1} - t_i = \tau^{\ep}$, it follows that
\begin{align*}
\|z \|_{L^\infty(   t_i , t_{i+1} ;  L^r(\TT^d))   }  &\leq  C_r \int_{t_i }^{t_{i+1}}\| R (t)\|_r    \, dt  + C_r \tau^{\ep} \delta(  \delta + \|u \|_{L^\infty_{t,x}} ).
\end{align*}

\end{proof}

\subsection{Temporal concentration by sharp gluing}
Since $u+v_i$ is an exact solution of the Navier-Stokes equations on each interval $[t_{i} , t_{i+1}]$,  $0 \leq i\leq \tau^{-\ep} - 1$, the next step is to suitably glue each $v_i$ together so that the glued solution $ u+ \sum \chi_i v_i$ is still an exact solution on a majority of the time interval $[0,1]$, with an error supported on many small disjoint sub-intervals.

We first choose cutoff functions that will be used to glue together $v_i$. We define $\chi_i \in C^\infty_c(\RR) $ be a smooth cutoff such that when $1  \leq i \leq \tau^{-\ep} -2 $, 
\begin{equation} \label{def:chi_i_1}
\chi_i = 
\begin{cases}
1 & \text{if  $t_i +   \tau   \leq t \leq t_{i+1} - \tau $ }\\
0& \text{if  $t_i + \tau/2 \geq t$ or $ t \geq t_{i+1} - \tau/2 $, }
\end{cases}
\end{equation}
and when $ i= 0 $, 
\begin{equation} \label{def:chi_i_2}
\chi_i = 
\begin{cases}
1 & \text{if  $0   \leq t \leq t_{i+1} - \tau $ }\\
0& \text{if  $ t \geq t_{i+1} - \tau/2 $,}
\end{cases}
\end{equation}
and when $ i= \tau^{-\ep}-1 $,
\begin{equation} \label{def:chi_i_3}
\chi_i = 
\begin{cases}
1 & \text{if  $t_i + \tau   \leq t \leq 1$ }\\
0& \text{if  $ t \leq t_{i} + \tau/2 $.}
\end{cases}
\end{equation}
In other words, we do not cut near the endpoints $t = 0 $ and $t=1$, and the glued solution $\overline{u}$ is an exact solution for a short time near $t=0$ and $t=1$. It is worth noting that in the iteration scheme $v_i$ for $i=0$ or $i=\tau^{-\ep}-1$ will be zero after the first step (since it is already an exact solution of \eqref{eq:NSE} there), so \eqref{def:chi_i_2} and \eqref{def:chi_i_3} are only used once. 

Furthermore, we require that the bounds
\begin{equation*}
|\nabla^m \chi_i | \lesssim_m \tau^{-m},
\end{equation*}
hold uniformly in $\tau$ and $i$.

Note that for sub-intervals $[t_i,t_{i+1}]$, $1 \leq i \leq \tau^{-\ep}-2 $, we cut near both the left and the right point of $[t_i , t_{i+1}]$. The left cutoff is to ensure smoothness near $t_i$ since each $v_i$ only has a limited amount of time regularity at $t=t_i$ whereas the right cutoff is where the intended gluing takes place.
With $\chi_i$ in hand, we can simply define the glued solution as
$$
\overline{u} : = u +   \sum_{i} \chi_i v_i =u + \overline{w}.
$$
It is clear that $\overline{u} : [0,1] \times \TT^d \to \RR^d$ has zero spatial mean and is also divergence-free. It remains to show that $\overline{u}$ satisfies the properties listed in Proposition \ref{prop:main_step_1}.

Heuristically, $\overline{u}$ should be an exact solution with a stress error supported on smaller intervals of size $ \tau$. To confirm this claim, we must compute the stress error $\overline{R}$ associated with $\overline{u}$. Since supports of $\chi_i $ are disjoint, we can compute
\begin{align*}
\p_t \overline{u} - \Delta \overline{u} + \D(\overline{u} \otimes \overline{u}) + \nabla p &= \D R + ( \p_t  -\Delta )\sum_{i} \chi_i v_i \\
& \quad + \sum_{i} \chi_i \D(u  \otimes  v_i  ) + \sum_{i} \chi_i \D( v_i  \otimes  u  ) \\
& \qquad + \sum_{i} \chi_i^2    \D(v_i \otimes  v_i  ).
\end{align*}
Thus, using the fact that $v_i$ solves \eqref{eq:concentrator} on $[t_i,t_{i+1}]$ and $u$ solves \eqref{eq:NSR} on $[0,1]$, we have
\begin{align*}
\p_t \overline{u} - \Delta \overline{u} + \D(\overline{u} \otimes \overline{u} ) + \nabla p &= \D R + \sum_{i}  \p_t  \chi_i  v_i  + \sum_{i} (\chi_i^2 -\chi_i )    \D(v_i \otimes  v_i  ) \\
& \qquad +  \sum_{i} \chi_i  \big( \p_t  v_i -\Delta  v_i +\D(v_i \otimes  v_i  )+  \D(u  \otimes  v_i  ) + \D(v_i  \otimes  u   ) \big) \\
& =  (1 - \sum_i \chi_i ) \D R + \sum_{i}  \p_t  \chi_i  v_i  + \sum_{i} (\chi_i^2 -\chi_i )    \D(v_i \otimes  v_i  ) - \sum_i  \chi_i \nabla q_i. 
\end{align*}

Now let 
\begin{equation}\label{eq:def_concentrated_stressR}
\overline{R} : = (1 - \sum_i \chi_i )   R + \mathcal{R} \sum_{i}  \p_t  \chi_i  v_i  +   \sum_{i} (\chi_i^2 -\chi_i )     v_i \mathring{\otimes}  v_i,
\end{equation}
where $ \mathring{\otimes}$ denotes a traceless tensor product, i.e. $f \mathring{\otimes} g = f_i    g_j  -  \frac{1}{d}\delta_{ij}  f_k g_k     $. Since each $v_i$ has zero spacial mean, $\D \mathcal{R}  v_i = v_i$, and we can conclude that
$$
\p_t \overline{u} - \Delta \overline{u} + \D(\overline{u} \otimes \overline{u} )  +\nabla \overline{p} = \D\overline{R},
$$
where the pressure $\overline{p} : [0,1]\times \TT^d \to \RR$ is defined by
$$
\overline{p} = p +\sum_{i} \chi_i q_i - \sum_i( \chi_i^2 -\chi_i) \frac{|v_i|^2}{d  } .
$$

The last step is then to show that the new Reynolds stress $\overline{R}$ is comparable to the original one when measured in $L^1_t$. It is clear that $\overline{R}$ is much more ``turbulent'' than the original $R$ as its value changes much more drastically due to the sharp cutoff $\chi$ near the endpoints of each interval $[t_i , t_{i+1} ]$.

The heuristic is that if $\tau$ is small enough, then $v_i$ behaves linearly with a rate $\sim \D R$, and thus gluing together $v_i$ only counts the input from the stress forcing $\D R$. More precisely, the leading order term in \eqref{eq:def_concentrated_stressR} is the second term, where $\mathcal{R} v_i$ is proportional to $ R $ thanks to Proposition \ref{estimates_on_vi}.
\begin{proposition}
For any $1< r< \infty$, there exists a universal constant $C_r $ depending on $ r,\ep$ and $d$ such that for all sufficiently small $\tau  > 0$, the glued solution $( \overline{u} ,  \overline{R})$ satisfies
$$
\|  \overline{R} \|_{L^1_t  L^r} \leq C_r \|  R  \|_{L^1_t  L^r} .
$$
\end{proposition}
\begin{proof}
By the triangle inequality, we need to estimate
$$
\|  \overline{R} \|_{L^1_t  L^r} \leq  \Big\|  (1 - \sum_i \chi_i )   R \Big\|_{L^1_t  L^r}  + \sum_{i}  \big\|   \p_t  \chi_i \mathcal{R} v_i \big\|_{L^1_t  L^r}  +  \Big\|  \sum_{i} (\chi_i -\chi_i^2 )     v_i \mathring{\otimes}  v_i     \Big\|_{L^1_t  L^r}.
$$
The idea is to treat the first and the last terms as small errors.  By H\"older's inequality, we get
$$
\|  \overline{R} \|_{L^1_t  L^r} \leq  \Big\|  1 - \sum_i \chi_i  \Big\|_{L^1([0,1])}  \| R \|_{L^\infty_t  L^r}  +   \sum_{i}  \|   \p_t  \chi_i \|_{L^1([0,1])}  \| \mathcal{R} v_i \|_{L^\infty_t  L^r}  +  \sum_{i} \|  (\chi_i -\chi_i^2 ) \|_{L^1([0,1])}   \|  v_i \mathring{\otimes}  v_i     \|_{L^\infty_t  L^r}.
$$
By the definition of the cutoff $\chi_i$, we have the following trivial bounds (with implicit constants depending on $\ep$) in time:
\begin{align*}
\Big\|  1 - \sum_i \chi_i  \Big\|_{L^1([0,1])} & \lesssim  \tau^{1 -\ep},  
\end{align*}
and for any $0 \leq  i \leq \tau^{-\ep}-1$,
\begin{align*}
\|   \p_t  \chi_i \|_{L^1([0,1])}  & \lesssim  1,   \\  
\|  \chi_i -\chi_i^2  \|_{L^1([0,1])}  & \lesssim  \tau. 
\end{align*}
In addition, it follows from the bounds for $v_i$ in Proposition~\ref{estimates_on_vi} that
\begin{align*}
\|  v_i \mathring{\otimes}  v_i     \|_{L^\infty_t  L^r} \lesssim \|  v_i   \|_{L^\infty_t  H^d}^2 \leq \delta^2.
\end{align*}

Combining these together and using the bound on $\mathcal{R} v_i$ from Proposition~\ref{estimates_on_vi},  we get
\begin{align*}
\|  \overline{R} \|_{L^1_t  L^r} & \lesssim  \tau^{1-\ep}  \| R \|_{L^\infty_t  L^r}  +   \sum_{i} \left(  \int_{t_i}^{t_{i+1}} \| R(t)\|_{  L^r} \, dt  +C_u \delta  \tau^{\ep}  \right)+  \delta^2 \sum_{i} \tau     \\
& \lesssim     \tau^{1-\ep}  \| R \|_{L^\infty_t  L^r}  + \| R \|_{L^1_t  L^r}  + C_u \delta, 
\end{align*}
where we have used the fact that $\sum_i 1 \leq \tau^{-\ep} \leq \tau^{-1}$.

The conclusion follows once we choose $\tau>0$ sufficiently small such that 
$$
 C_u \delta  \leq \| R \|_{L^1_t  L^r} \quad \text{and} \quad \tau^{1-\ep}  \| R \|_{L^\infty_t  L^r}  \leq \| R \|_{L^1_t  L^r}.
$$
\end{proof}

\subsection{Proof of Proposition \ref{prop:main_step_1}}

We conclude this section with the last step in the proof of Proposition \ref{prop:main_step_1}. Since all the estimates have been obtained, we only need to verify that the temporal support of $\overline{w} = \sum_i \chi_i v_i$ is contained in $\tilde{I}$ and show the well-preparedness of $( \overline{u} , \overline{R})$.

Note that $(u,R)$ is well-prepared for $\tilde{I}$ and $\tilde{\tau}$, and it follows from \eqref{eq:concentrator} that
$$
v_i \equiv 0 \quad \text{for $0 \leq i \leq \tau^{-\ep}-1$ such that $R\equiv 0$ on $[t_i, t_{i+1}]$}.
$$
Hence, if $\tau^{\ep}= |[t_i, t_{i+1}]| $ is sufficiently smaller than $  \tilde{\tau}$, the definition of well-preparedness of $(u,R)$ implies that
$$
\bigcup_i \Supp_t \chi_i  v_i \subset \tilde{I}.
$$
Thus we have proved that $ \Supp_t  \overline{w} \subset \tilde{I} $ .

Let us now show the well-preparedness of $( \overline{u} , \overline{R})$. Define an index set 
$$
E = \{ i\in \ZZ: 1   \leq i \leq    \tau^{-\ep} -1 \,\, \text{and} \,\, v_i \not \equiv 0 \} 
$$ satisfying a trivial estimate
\[
|  E  | \leq\tau^{-\ep}.
\]
The idea is that the concentrated stress $\overline{R}$ is supported around each $t_i$, for $i\in E$. Therefore, we can define a set on the time axis
\[
{I} : = \bigcup_{i \in E } \Big[t_i -  \frac{5\tau}{2} ,t_i+ \frac{5\tau}{2} \Big]  ,
\]
where as before $t_i = i \tau^{\ep}$. Note that each interval in $I$ has length $ 5\tau$ and the total number of intervals is at most $  \tau^{-\ep}$, consistent with the well-preparedness, and $0,1 \not \in I$ due to \eqref{def:chi_i_2}

Now take any $t \in [0,1]$ such that $\dist(t, I^c) \leq  \frac{3\tau}{2} $. Then by \eqref{def:chi_i_1} and  \eqref{def:chi_i_2},
\[
\sum_{i  }\chi_i(t) = 1.
\]
Moreover, $\p_t  \chi_i(t)=0$ and $\chi_i(t) \in\{0, 1\}$ for any $i$.
Consequently, 
\[
\overline{R} (t)  = (1 - \sum_i \chi_i )   R + \mathcal{R} \sum_{i}  \p_t  \chi_i  v_i  + \sum_{i} (\chi_i^2 -\chi_i )     v_i \mathring{\otimes}  v_i    =0,
\]
for every $t$ such that $\dist(t, I^c) \leq \frac{3\tau}{2} $. In particular, $(\overline{u},\overline{R})$ is also well-prepared, which concludes the proof.

%%%%%%%%%%%%%%%%%%%%%%%%%%%%%%%%%%%%%%%%%%%%%%%%%%%%%%%%%%%%%%%%%%%%%%%%%%%%%%%%%%%%%%%%%%%%%%%%%%%%%%%%%%%%%%%%%%%%%%%%%%%%%%
\section{Convex integration in space-time}\label{sec:proof_step_2_convex_integration}
%%%%%%%%%%%%%%%%%%%%%%%%%%%%%%%%%%%%%%%%%%%%%%%%%%%%%%%%%%%%%%%%%%%%%%%%%%%%%%%%%%%%%%%%%%%%%%%%%%%%%%%%%%%%%%%%%%%%%%%%%%%%%%

In this section, we will use a convex integration scheme to reduce the size of the Reynolds stress. The goal is to design a suitable velocity perturbation $w$ to the glued solution $( \overline{u}, \overline{R})$ so that 
$$
u_1 : =\overline{u} + w
$$
solves the equation \eqref{eq:NSR} with a much smaller Reynolds stress $ R_1$.

The main goal of the current and the following section is summarized in the following proposition.
\begin{proposition}\label{prop:main_2}
There exists a universal constant $M >0$ such that for any $p < 2$ and $q<\infty$, there exists $r>1$ depending only on $p,q$ and $d$ such that the following holds.

Let $\delta>0$ and  $(\overline{u},\overline{R})$ be a well-prepared smooth solution of \eqref{eq:NSR} given by Proposition \ref{prop:main_step_1} for the set $ I  $ and time scale $\tau $. Then there exists another well-prepared smooth solution $(u_1,R_{1})$  of \eqref{eq:NSR} for the same set $I    $ and time scale $\tau  $ such that 
$$
\|R_{1}  \|_{L^1(0,1; L^r( \TT^d  ))} \leq \delta .
$$
Moreover, the velocity perturbation  $w : = u_{1}  - \overline{u} $ satisfies
\begin{enumerate}
\item  The temporal support of $w$ is contained in $I$, i.e.
$$
\Supp w \subset I \times \TT^d;
$$

\item The $L^2_{t,x} $ estimate,
$$
 \|w \|_{L^2([0,1] \times \TT^d )} \leq M  \| \overline{R} \|_{L^1 ([0,1] \times \TT^d )} ;
$$

\item The $L^{p}_t L^\infty   \cap L^1_t W^{1,q}$  estimate,
$$
\|w \|_{L^{p }(0,1;   L^\infty (\TT^d))} + \|w \|_{L^{1 }(0,1;   W^{1,q} (\TT^d))} \leq    \delta .
$$

\end{enumerate} 

\end{proposition}

It is clear that Proposition~\ref{prop:main} follows from Proposition~\ref{prop:main_step_1} and Proposition~\ref{prop:main_2}. In the remainder of this section, we will construct the velocity perturbation $w$ and define its associated Reynolds stress $R_1$ and pressure $p_1$. The well-preparedness of $(u_1, R_1)$ will be an easy consequence of the definition of $w$, whereas all the estimates will be proven in the next section.

\subsection{Stationary Mikado flows for the convex integration}
The main building blocks of the convex integration scheme is the Mikado flows $\mathbf{W}_{k}:\TT^d \to \RR^d $ introduced in \cite{MR3614753}. In other contexts \cite{MR3884855,MR3951691}, they are called concentrated Mikado flows, or Mikado flows with concentration since they are supported on periodic cylinders with a small radius. Here for brevity, we simply refer to them as the Mikado flows.

We start with a geometric lemma that dates back to the work of Nash \cite{MR0065993}. A proof of the following version, which is essentially due to De Lellis and Székelyhidi Jr., can be found in \cite[Lemma~3.3]{MR3340997} and \cite[Lemma~3.2]{MR3090182}. This lemma is the key reason to use Mikado flows in the convex integration.

Recall that $ \mathcal{S}^{d \times d}_+$ is the set of positive definite symmetric $d\times d$ matrices and $\ek = \frac{k}{|k|}$ for any $k\in \ZZ^d$.

\begin{lemma}\label{lemma:geometric}
For any compact subset $\mathcal{N} \subset  \mathcal{S}^{d \times d}_+$, there exists a finite set $\L \subset \ZZ^d$ and smooth functions $\Gamma_k \in C^\infty(\mathcal{N}; \RR )$ for any $k \in \L$ such that 
\begin{align*}
R = \sum_{k \in \L } \Gamma_k^2(R) \ek \otimes \ek, \qquad \text{for all } \, R \in \mathcal{N} .
\end{align*}
\end{lemma}

We apply the lemma for $\mathcal{N}= B_{\nicefrac{1}{2}}(\Id)$
where $ B_{\nicefrac{1}{2}}(\Id)$ denotes the metric ball of radius $1/2$ around the identity Id in the space $ \mathcal{S}^{d \times d}_+$ to obtain smooth functions $\Gamma_k  $ for any $k \in \L \subset \mathbb{Z}^d$. Throughout the paper, the direction set $\L$ is fixed, and we construct the Mikado flows as follows.

We choose points $p_k \in (0,1)^d $ such that $p_k \neq p_{-k} $ if both $k,-k \in \L$. For each $k\in \L $, we denote by $l_k \subset \TT^d$ the periodic line passing through $p_k$ in direction $k$, namely
$$
l_k =\{    t k + p_k \in \TT^d  : t \in \RR  \}.
$$

Since $\L$ is a finite lattice set and we identify $\TT^d $ with a periodic box $  [0,1]^d$, there exists a geometric constant $C_\L \in \NN$ depending on the set $\L$  such that 
$$
|l_k \cap l_{k'}| \leq C_\L \quad \text{for any $k, k' \in \L$},
$$
where we note that  $ l_k \cap l_{-k} = \emptyset$ due to $p_k \neq p_{-k} $.

Here we do not even require nonparallel periodic lines to be disjoint in $d \geq 3$ in contrast to \cite{MR3951691}. Since nonparallel periodic lines have to intersect in the 2D case, we would rather present a unified approach based on the fact that the intersection parts are very small due to the concentration parameter $\mu>0$, see Theorem \ref{thm:main_thm_for_W_k} below.

Let $  \mu > 0$ be the spatial concentration parameter whose value will be fixed in the end of the proof. Let $ \phi,\psi \in C^\infty_c ( [1/2 ,1])$ and constants $c_k>0$ be such that for any sufficiently large $\mu $ if we define $\psi_k, \phi_k : \TT^d \to \RR$ by
\begin{equation}\label{eq:def_psi_k}
\psi_k := c_k \mu^{\frac{d-1}{2}} \psi(\mu \dist( l_k, x))
\end{equation}
and
\begin{equation}\label{eq:def_phi_k}
\phi_k := c_k \mu^{\frac{d-1}{2} -2} \phi(\mu \dist( l_k, x)) 
\end{equation}
then 
\begin{equation}\label{eq:def_phi_k_psi_k}
\Delta \phi_k = \psi_k \quad \text{on $\TT^d$} \quad \text{and} \quad \int_{\TT^d} \psi_k^2 \, dx =1.
\end{equation}

Note that
\begin{equation}\label{eq:small_intersection_l_k}
\Supp \psi_k \cap \Supp \psi_{k'} \subset \{ x\in \TT^d: \dist( x, l_k \cap  l_{k'} )\leq M_\L \mu^{-1})
\end{equation}
for a sufficiently large constant $M_\L$ depending on $\L$.

Finally, the stationary Mikado flows $\mathbf{W}_{k}:\TT^d \to \RR^d$ are defined by
\begin{equation}\label{eq:def_bf_W_k}
\mathbf{W}_{k} = \psi_k \ek,
\end{equation}
where the constant vector $\ek = \frac{k}{|k|}$.
Using the gradient field $  \nabla \phi_k$, we may write $\bwk $ as a divergence of a skew-symmetric tensor $\mathbf{\Omega}_k \in C^\infty_0(\TT^d, \RR^{d\times d})$,
\begin{equation}\label{eq:def_bf_O_k}
\mathbf{\Omega}_k := \ek \otimes \nabla \phi_k - \nabla \phi_k  \otimes \ek.
\end{equation}
Indeed, $\bok$ is skew-symmetric by definition, and by a direct computation
$$
\D \bok = \D ( \nabla \phi_k )\ek -( \ek \cdot \nabla  )\nabla \phi_k = \Delta \phi_k \ek - 0 = \bwk.
$$

We summarize the main properties of the Mikado flows $\bwk$ in the following theorem.
\begin{theorem} \label{thm:main_thm_for_W_k}
Let $d \geq 2$ be the dimension. The stationary Mikado flows  $\mathbf{W}_{k}:\TT^d \to \RR^d$ satisfy the following.
\begin{enumerate}
    \item Each $ \mathbf{W}_{k} \in C^\infty_0(\TT^d)$ is divergence-free,  satisfies
    $$
    \mathbf{W}_{k}  =\D \mathbf{\Omega}_k,
    $$
    and solves the pressureless  Euler equations
    $$
    \D(  \bwk \otimes  \bwk ) =0.
    $$

    \item For any $1\leq p \leq \infty$, the estimates
    \begin{align*}
\mu^{-m}  \|  \nabla^m \bwk  \|_{L^p(\TT^d)}   &\lesssim_m  \mu^{\frac{d-1}{2}  - \frac{d-1}{p}},\\
   \mu^{-m}\|  \nabla^m  \mathbf{\Omega}_k \|_{L^p(\TT^d)}   &\lesssim_m  \mu^{-1+\frac{d-1}{2}  - \frac{d-1}{p}},
\end{align*}
hold uniformly in $  \mu$.
    
    \item For any $ k\in \L$, there holds
    $$
    \fint_{\TT^d} \mathbf{W}_{ k}\otimes \mathbf{W}_{ k} = \ek  \otimes \ek, 
    $$
    and for any $1 \leq p \leq \infty$
    $$
    \big\| \mathbf{W}_{ k}\otimes \mathbf{W}_{ {k'}}\big\|_{L^p(\TT^d)} \lesssim \mu^{ (d-1)-\frac{d}{p}} \quad \text{if $k \neq k'$}.
    $$

\end{enumerate}
\end{theorem}
\begin{proof}
We only prove the last claim as the rest of them are standard and can easily be deduced from \eqref{eq:def_psi_k}-\eqref{eq:def_phi_k_psi_k}. It suffices to assume $k \neq -k'$. Using the $L^\infty$ bound in (2), we obtain
\begin{align*}
 \big\| \mathbf{W}_{ k}\otimes \mathbf{W}_{ {k'}}\big\|_{L^p(\TT^d)} \lesssim \mu^{ (d-1)}  \Big( \int_{\Supp \psi_k \cap \Supp \psi_{k'} }  \, dx\Big)^{\frac{1}{p}}.
\end{align*}

Note that due to \eqref{eq:small_intersection_l_k} $  \Supp \psi_k \cap \Supp \psi_{k'}$ is contained in a union of  finitely many balls of radius $\sim \mu^{-1}$. Thus
\begin{align*}
 \big\| \mathbf{W}_{ k}\otimes \mathbf{W}_{ {k'}}\big\|_{L^p(\TT^d)} \lesssim \mu^{ (d-1)  - \frac{d}{p}}   .
\end{align*}

\end{proof}

\subsection{Implementation of temporal concentration}
Since Mikado flows are stationary, the velocity perturbation will be homogeneous in time if we simply use Lemma~\ref{lemma:geometric} to define the coefficients. To obtain $L^p_t L^\infty $ and $L^1_t W^{1,q}$ estimates, it is necessary to introduce temporal concentration  in the perturbation.

To this end, we choose temporal functions $ g_{\kappa}$ and $h_\kappa$ to oscillate the building blocks $ \mathbf{W}_k$ intermittently in time. Specifically, $ g_{\kappa}$ will be used to oscillate $ \mathbf{W}_k$ so that the space-time cascade balances the low temporal frequency part of the old stress error $\overline{R} $, whereas $h_\kappa$ is used to define a temporal corrector whose time derivative will further balance the high temporal frequency part of the old stress error $\overline{R} $ .

Let nonnegative $g \in C_c^\infty((0,1))$ be such that
$$
\int_{0}^1 g^2(t) \, dt =1.
$$
To add in temporal concentration, let $\kappa>0$ be a large constant whose value will be specified later and define $g_\kappa : [0,1] \to \RR$ as the $1$-periodic extension of  $\kappa^{1/2} g (\kappa t)$ so that
$$
\| g_\kappa \|_{L^p([0,1])} \lesssim \kappa^{\frac{1}{2} - \frac{1}{p}}\quad \text{for all $1\leq p \leq \infty$}.
$$

The value of $\kappa$ will be specified later and the function $g_\kappa$ will be used in the definition of the velocity perturbation. As we will see in Lemma \ref{lemma:R_osc_decomposition}, the nonlinear term can only balance a portion of the stress error $\overline{R}$ and there is a leftover term which is of high temporal frequency. This motivates us to consider the following temporal corrector.

Let $h_\kappa :[0,1] \to \RR $ be defined by
$$
h_\kappa(t)  = \int_0^t g_\kappa^2(s) -1 \, ds.
$$
In view of the zero-mean condition for $g_\kappa^2(t) -1$, the function $h_\kappa :[0,1] \to \RR $ is well-defined and periodic, and we have
\begin{equation}\label{eq:bound_on_h_k}
\| h_\kappa \|_{L^\infty([0,1])} \leq 1,
\end{equation}
uniformly in $\kappa$.

We remark that for any $\nu \in \NN$, the periodically rescaled function $g_\kappa( \nu \cdot ): [0,1] \to \RR$ also verifies the bound
\begin{equation} \label{e:bound_on_g_k}
\| g_\kappa( \nu \cdot) \|_{L^p([0,1])} \lesssim \kappa^{\frac{1}{2} - \frac{1}{p}} \quad \text{for all $1\leq p \leq \infty$}.
\end{equation}
Moreover, we have the identity
\begin{equation} \label{eq:time_derivative_of_h_k}
 \p_t \left( \nu^{-1} h_{\kappa}(\nu t ) \right) = g_\kappa^2( \nu t ) -1,
\end{equation}
which will imply the smallness of the corrector, cf \eqref{eq:def_w_t}.

\subsection{Space-time cutoffs}

Before introducing the velocity perturbation, we need to define two important cutoff functions, one to ensure Lemma~\ref{lemma:geometric} applies and the other to ensure the well-preparedness of the new solution $(u_1,R_1)$.

Since Lemma~\ref{lemma:geometric} is stated for a fix compact set in $\mathcal{S}_+^{d\times d}$, we need to introduce a cutoff for the stress $\overline{R}$. Let $\chi: \RR^{d \times d }  \to \RR^+$ be a positive smooth function such that $\chi $ is monotonically increasing with respect to $|x|$ and
\begin{equation}\label{eq:def_chi}
\chi (x) =
\begin{cases}
4 \| \overline{R}\|_{L^1([0,1]\times \TT^d)} & \text{if $0 \leq |x| \leq \| \overline{R}\|_{L^1([0,1]\times \TT^d)}$}\\
4|x| & \text{if $|x| \geq 2\| \overline{R}\|_{L^1 ( [0,1]\times \TT^d ) }$}.
\end{cases}
\end{equation}

With this cutoff $\chi$, we may define a divisor for the stress $\overline{R}$ so that Lemma \ref{lemma:geometric} applies. Indeed, define $\rho \in  C^\infty([0,1] \times \TT^d)$ by
\begin{equation}\label{eq:def_rho}
\rho =  \chi(\overline{R})  .
\end{equation}
Then immediately by \eqref{eq:def_chi},
$$
\Id - \frac{\overline{R}}{\rho} \in B_{\frac{1}{2}}(\Id) \quad \text{for any $(t,x) \in [0,1] \times \TT^d$},
$$
which means we can use $\Id - \frac{\overline{R}}{\rho}$ as the argument in the smooth functions $\Gamma_k$ given by Lemma \ref{lemma:geometric}.

Next, we need another cutoff to take care of the well-preparedness of the new solution $(u_1, R_1) $ as the perturbation has to be supported within $ I$. Let $\theta \in C^\infty_c(\RR)$ be a smooth temporal cutoff function such that
\begin{equation}\label{eq:def_theta}
\theta(t) = 
\begin{cases}
1 & \text{if $\dist(t, I^c ) \geq \frac{3\tau}{2}$   } \\
0 & \text{if $\dist(t, I^c ) \leq  \tau$,}
\end{cases}
\end{equation}
where $I \subset [0,1]$ and $\tau>0$ are given by Proposition \ref{prop:main_step_1}. Note that this cutoff ensures that the new solution will still be well-prepared.

\subsection{The velocity perturbation}
We recall that we have defined four parameters for the perturbation so far:
\begin{enumerate}
    \item Temporal oscillation $ \nu \in \NN$  and  temporal concentration $ \kappa >0 $;
    \item Spatial oscillation $ \sigma \in \NN$  and spatial concentration  $ \mu \in \NN$.
    
\end{enumerate}
These four parameters are assumed to be sufficiently large for the moment and will be taken to be explicit fixed powers of a frequency parameter $\l>0$ in the next section, where we also fix the value of $r>1$ appearing in the main proposition.

With all the ingredients in hand, we are ready to define the velocity perturbation. In summary, the velocity perturbation $w:[0,1] \times \TT^d \to \RR^d$ consists of three parts,
$$
w = w^{(p)} + w^{(c)} +w^{(t)} .
$$

The principle part of the perturbation $w^{(p)}$ consists of super-positions of the building blocks $\bwk  $ oscillating with period $\sigma^{-1}$ on $\TT^d$ and period $\nu^{-1}$ on $[0,1]$:
\begin{equation}\label{eq:def_w_p}
w^{(p)} (t,x ): = \sum_{k \in \L } a_k(t,x ) \bwk (\sigma x ),
\end{equation}
where the amplitude function $a_k : [0,1] \times \TT^d \to \RR$ is given by
\begin{equation}\label{eq:def_a_k}
a_k =  \theta  g_{\kappa }(\nu t)  \rho^{\nicefrac{1}{2}} \Gamma_k\Big(\Id - \frac{ \overline{R}}{\rho }\Big).
\end{equation}
Note that \eqref{eq:def_w_p} is not divergence-free. To fix this, we introduce a divergence-free corrector using the tensor potential $ \bok$
\begin{equation}\label{eq:def_w_c}
w^{(c)} (t,x ): =      \sigma^{-1}\sum_{k \in \L }   \nabla  a_k(t,x ) : \bok(\sigma x )     . 
\end{equation}
Indeed,  we can rewrite $ w^{(p)}  + w^{(c)}$ as
\begin{equation}\label{eq:w_p_plus_w_c}
\begin{aligned}
 w^{(p)}  + w^{(c)}   & =  \sigma^{-1} \sum_{k \in \L } a_k(t,x )\D \bok  (\sigma x ) +  \sigma^{-1}  \sum_{k \in \L }   \nabla  a_k(t,x ) : \bok(\sigma x ) \\
 &=  \sigma^{-1}\D\sum_{k \in \L } a_k(t,x ) \bok  (\sigma x ),
\end{aligned}
\end{equation}
where  each $a_k \bok $ is skew-symmetric and hence $ \D ( w^{(p)}  + w^{(c)}) = 0 $.

Finally, we define a temporal corrector to balance the high temporal frequency part of the interaction. This ansatz was first introduced in \cite{MR3898708} and also used in \cite{1809.00600}. The heart of the argument is to ensure that
$$
\p_t w^{(t)}+ \D(w^{(p)} \otimes w^{(p)} ) = \text{Pressure gradient}+  \text{Terms with high spacial frequencies} +\text{Lower order terms}.
$$
However, the key difference between \cite{MR3898708,1809.00600} and the current scheme is that here the smallness of the corrector is free and it does not require much temporal oscillation, which is the reason we must use stationary spatial building blocks.

Specifically, the temporal corrector $w^{(t)}$ is defined as
\begin{equation}\label{eq:def_w_t}
w^{(t)} = \nu^{-1}  h_\kappa(\nu t)\left( \D(\overline{R} ) - \nabla\Delta^{-1}\D \D (\overline{R} ) \right),
\end{equation}
where we note that $\Delta^{-1}$ is well-defined on $\TT^d$ since $\D \D (\overline{R} ) = \p_i \p_j  \overline{R}_{ij} $ has zero spatial mean.

It is easy to check $\Supp_t w^{(t)} \subset \Supp_t \overline{R} $ and $ w^{(t)}$ is divergence-free. Indeed,
$$
\D w^{(t)}  = \nu^{-1}  h_\kappa(\nu t)\left( \p_i \p_j  \overline{R}_{ij} - \p_k \p_k\Delta^{-1} \p_i \p_j  \overline{R}_{ij} ) \right) =0.
$$

In the lemma below, we show that the leading order interaction of the principle part $w^{(p)}$ is able to balance the low temporal frequency part of the stress error $ \overline{R} $, which motivates the choice of the corrector  $w^{(t)}$, cf. \eqref{eq:New_Reynolds_Stress_2}.
\begin{lemma}\label{lemma:a_k_interactions}
The coefficients $a_k $ satisfy
$$
  a_k  = 0 \quad \text{if} \quad   \dist(t, I^c) \leq \tau,
$$
and
$$
 \sum_{k \in \L} a_k^2 \fint_{\TT^d}  \bwk \otimes \bwk \, dx = \theta^2 g_{\kappa }^2(\nu t) \rho \Id  - g_{\kappa }^2(\nu t) \overline{R}.
$$
\end{lemma}
\begin{proof}
The first property follows from \eqref{eq:def_theta}.

Now since $ \fint  \bwk \otimes \bwk = \ek \otimes \ek$, a direct computation and Lemma~\ref{lemma:geometric} give
$$
\sum_{k \in \L } a_k^2 \ek \otimes \ek = \sum_{k \in \L } \theta^2 g_{\kappa }^2(  \nu t) \rho \Gamma_k^2 \ek \otimes \ek = \theta^2 g_{\kappa }^2(  \nu t) \rho \sum_{k \in \L } \Gamma_k^2\Big(\Id - \frac{ \overline{R} }{\rho }\Big) \ek \otimes \ek = \theta^2 g_{\kappa }^2(  \nu t) \rho \Id  - \theta^2g_{\kappa }^2( \nu t) \overline{R}.
$$
The identity follows from the fact that 
$$
\Supp \overline{R} \subset  \{t: \theta(t)=1 \}.
$$
\end{proof}

\subsection{The new Reynolds stress}
In this subsection, our goal is to design a suitable stress tensor $R_1 : [0,1] \times \TT^d \to \mathcal{S}^{d  \times  d}_0 $ such that the pair $(u_1, R_1)$ is a smooth solution of \eqref{eq:NSR} for a suitable smooth pressure $p_1$.

We first compute the nonlinear term and isolate nonlocal interactions:
\begin{align} \label{eq:New_Reynolds_Stress}
\D(w^{(p)} \otimes w^{(p)} +\overline{R}   ) = \D \Big[ \sum_{k \in \L} a_k^2 \bwk(\sigma x) \otimes \bwk(\sigma x)  +\overline{R}     \Big]  + \D R_{\Far }, 
\end{align}
where $R_{\Far }$ denotes the nonlocal interactions between Mikado flows of different directions
\begin{equation}
R_{\Far } = \sum_{k \neq k'} a_k a_{k'}  \bwk(\sigma x) \otimes \bw_{k'}(\sigma x).
\end{equation}
And then we proceed to examine the first term in \eqref{eq:New_Reynolds_Stress}, for which by Lemma \ref{lemma:a_k_interactions} we have
\begin{align}
\D \Big[ \sum_{k \in \L} &a_k^2 \bwk(\sigma x) \otimes \bwk(\sigma x)  +\overline{R}     \Big] \nonumber \\
& =  \D\Big[ \sum_{k \in \L} a_k^2 \big( \bwk(\sigma x) \otimes \bwk(\sigma x)   -\fint  \bwk \otimes \bwk   +\fint  \bwk \otimes \bwk  \big) + \overline{R}      \Big]   \nonumber \\
&  =  \D \Big[ \sum_{k \in \L} a_k^2 \big( \bwk(\sigma x) \otimes \bwk(\sigma x)   -\fint  \bwk \otimes \bwk  \big) \Big] +  \nabla ( \theta^2 g_k^2 \rho   )  + (1 - g_{\kappa }^2(\nu t) )\D \overline{R}. \label{eq:New_Reynolds_Stress_01} 
\end{align}
Finally, using the product rule, we compute the divergence term as
\begin{align}
\D \sum_{k \in \L} &a_k^2 \big( \bwk(\sigma x) \otimes \bwk(\sigma x)   -\fint  \bwk \otimes \bwk  \big) \nonumber \\ 
&=\sum_{k \in \L} \nabla(a_k^2 ) \cdot  \big( \bwk(\sigma x) \otimes \bwk(\sigma x)   -\fint  \bwk \otimes \bwk  \big)      . \label{eq:New_Reynolds_Stress_02} 
\end{align}
Typical in the convex integration, we can gain a factor of $\sigma^{-1}$ in \eqref{eq:New_Reynolds_Stress_02}  by inverting the divergence. To this end, let us use the bilinear anti-divergence operator $\mathcal{B}$ defined in Appendix \ref{sec:append_tech}. Since \eqref{eq:New_Reynolds_Stress_02} has zero spatial mean,  by \eqref{eq:div_B_tensor} it is equal to $\D R_{\Osc,x} $, where
\begin{equation}\label{eq:New_Reynolds_Stress_03}
R_{\Osc,x} =  \sum_{k \in \L} \mathcal{B}\Big(\nabla(a_k^2 ) ,    \bwk(\sigma x) \otimes \bwk(\sigma x)  -\fint \bwk \otimes \bwk \Big).
\end{equation}
Combining \eqref{eq:New_Reynolds_Stress}, \eqref{eq:New_Reynolds_Stress_01}, \eqref{eq:New_Reynolds_Stress_02}, and \eqref{eq:New_Reynolds_Stress_03} we have
\begin{equation}\label{eq:New_Reynolds_Stress_2}
\D(w^{(p)} \otimes w^{(p)} +\overline{R}   ) =  \D R_{\Osc,x} + \D R_{\Far } +  \nabla ( \theta^2 g_k^2(\nu t) \rho   )  + (1 - g_{\kappa }^2(\nu t) )\D \overline{R}.
\end{equation}

In view of the above computations, we define a temporal oscillation error
\begin{equation}\label{eq:def_osc_t_error}
R_{\Osc,t} = \nu^{-1} h_k( \nu t)     \p_t \overline{R}  ,
\end{equation}
so that the following decomposition holds.

\begin{lemma} \label{lemma:R_osc_decomposition}
Let the space-time oscillation error $R_{\Osc}$ be
$$
R_{\Osc} = R_{\Osc,x}   + R_{\Osc,t} + R_{\Far}  .
$$
Then
$$
\p_t w^{(t)}  + \D(w^{(p)} \otimes w^{(p)} +\overline{R}  ) + \nabla P = \D R_{\Osc}  . 
$$
where the pressure term $P$ is defined by
\begin{equation}\label{eq:def_pressure_P}
P = - \theta^2g_k^2(\nu  t) \rho  - \nu^{-1}   \Delta^{-1} \D \D  \p_t \big(\overline{R}h_\kappa( \nu t)    \big)  .  
\end{equation}
\end{lemma}
\begin{proof}
By the definition of $ w^{(t)}$, we have
\[
\begin{split}
\p_t w^{(t)} &= \p_t\big(\nu^{-1} h_\kappa( \nu t)  \D   \overline{R}    \big) - \p_t\big(\nu^{-1} h_\kappa( \nu t)    \nabla \Delta^{-1} \D \D  \overline{R}  \big) \\
&=(g_\kappa^2(\nu t ) -1)   \D   \overline{R}   +  \nu^{-1} h_\kappa( \nu t)   \D   \p_t \overline{R} - \nu^{-1} \nabla \Delta^{-1} \D \D  \p_t \big(\overline{R}h_\kappa( \nu t)    \big)   \\
&=(g_\kappa^2(\nu t ) -1)   \D   \overline{R}   + R_{\Osc,t} - \nu^{-1} \nabla \Delta^{-1} \D \D  \p_t \big(\overline{R}h_\kappa( \nu t)    \big)  ,
\end{split}
\]
where we used identity~\eqref{eq:time_derivative_of_h_k} for the time derivative of $h_{\kappa}$. The conclusion follows immediately from \eqref{eq:New_Reynolds_Stress_2}.

\end{proof}

Finally, we can define the correction error and the linear error as usual:
\begin{equation}\label{eq:def_R_cor}
R_{\Cor} =  \mathcal{R}  \Big( \D \big( (w^{(c)}+w^{(t)})  \otimes w \big)  + \D \big( w^{(p)} \otimes ( w^{(c)} + w^{(t)} ) \big)  \Big)
\end{equation}
and
\begin{equation}\label{eq:def_R_lin}
R_{\Lin} = \mathcal{R}\Big(  \p_t (w^{(p)} +w^{(c)}  ) - \Delta  w  +  \D \big( \overline{u} \otimes w + w \otimes  \overline{u}  \big)  \Big) ,
\end{equation}
where $\mathcal{R}$ is an inverse divergence operator defined in \eqref{eq:appendix_R_def}.

To conclude, we summarize the main results in this section below.
\begin{lemma}\label{lemma:new_stress_R_1}
Define the new Reynolds stress by
$$
R_1 =  R_{\Lin} +R_{\Cor} + R_{\Osc} ,
$$
and the new pressure by
$$
p_1 =   \overline{p} +P.
$$
Then $(u_1,R_1)$ is a well-prepared solution to \eqref{eq:NSR},
$$
\p_t u_1 -\Delta u_1 + \D(u_1 \otimes u_1 ) + \nabla p_1 = \D R_1,
$$
and the velocity perturbation $w = u_1 - \overline{u}$ satisfies
$$
\Supp w \subset I \times \TT^d,
$$
where $I$ is as in the well-preparedness of $(u_1 , R_1)$.
\end{lemma}
\begin{proof}
A direct computation of the left-hand side gives
\begin{align*}
\p_t u_1 -\Delta u_1 + \D(u_1 \otimes u_1 ) + \nabla p_1  =& \p_t \overline{u} -\Delta \overline{u}  + \D( \overline{u} \otimes \overline{u} ) + \nabla p\\
& +  \p_t w -\Delta w + \D(\overline{u}  \otimes w ) + \D(w  \otimes \overline{u} ) +\D(w  \otimes w ) + \nabla P  \\
= & \D \overline{R}+  \p_t w -\Delta w + \D(\overline{u}  \otimes w ) + \D(w  \otimes \overline{u} ) +\D(w  \otimes w ) + \nabla P,
\end{align*}
where we have use the fact that $(\overline{u} , \overline{R})  $ solves \eqref{eq:NSR} with pressure $\overline{p}$.

From the definitions \eqref{eq:def_R_cor}, \eqref{eq:def_R_lin} and Lemma \ref{lemma:R_osc_decomposition} we can conclude that $(u_1, R_1)$ solves \eqref{eq:NSR}.

The claim that $(u_1, R_1)$ is well-prepared and $\Supp_t w \subset I $ follows from the fact that $(\overline{u} , \overline{R})$ is given by Proposition \ref{prop:main_step_1} and the perturbations $w^{(p)}, w^{(c)}$ and  $w^{(t)}$ satisfy 
$$
  w^{(p)} = w^{(c)} =0 \quad \text{if $\dist(t, I^c) \leq  \tau$}
$$
and
$$
\Supp_t w^{(t)} \subset   \Supp_t \overline{R}  .
$$
\end{proof}

%%%%%%%%%%%%%%%%%%%%%%%%%%%%%%%%%%%%%%%%%%%%%%%%%%%%%%%%%%%%%%%%%%%%%%%%%%%%%%%%%%%%%%%%%%%%%%%%%%%%%%%%%%%%%%%%%%%%%%%%%%%%%%
\section{Proof of Proposition \ref{prop:main_2}}\label{sec:proof_step_2}
%%%%%%%%%%%%%%%%%%%%%%%%%%%%%%%%%%%%%%%%%%%%%%%%%%%%%%%%%%%%%%%%%%%%%%%%%%%%%%%%%%%%%%%%%%%%%%%%%%%%%%%%%%%%%%%%%%%%%%%%%%%%%%

In this section we will show that the velocity perturbation $w$ and the new Reynolds stress $R_1$ derived in Section \ref{sec:proof_step_2_convex_integration} satisfy the claimed properties in Proposition \ref{prop:main_2}.

As a general note, we use a constant $C_u$ for dependency on the previous solution $(\overline{u} , \overline{R})$ throughout this section. Unless otherwise indicated, in the statement of below lemmas and propositions the exponents $  p,q $ and $r$ refer to the ones given by Proposition \ref{prop:main}, cf. \eqref{eq:value_gamma} and \eqref{eq:def_of_r}.

\subsection{Choice of parameters}\label{subsection:choice_of_parameters}
We fix two parameters $0<\gamma<1$ and $1<r<2$ as follows.

\begin{enumerate}
\item First, we choose $ \gamma>0$ small enough such that
\begin{equation}\label{eq:value_gamma}
 10 d\gamma \leq \min\Big\{   \frac{1}{p} -\frac{1}{2},   \frac{1}{q}  \Big\} .
\end{equation}

\item Once $\gamma$ is fixed, choose $r > 1$ such that
\begin{equation}\label{eq:def_of_r}
 d   - \frac{d }{r} \leq  \gamma .
\end{equation}

\end{enumerate}
It is clear that $r>1$ only depends on $p,q$, and $d$ as claimed in Proposition \ref{prop:main_2}. Without loss of generality, we assume $q \geq 2$.

Let $\lambda$  be a sufficiently large number whose value will be fixed in the end. We choose the parameters $\sigma, \kappa$ along with $ \nu,\mu$ in the building blocks as explicit powers of $\lambda$ as follows.
\begin{enumerate}

    \item Temporal oscillation $\nu\in \NN$ and spatial  oscillation  $\sigma  \in \NN$:
    \begin{align*}
    \nu &= \lceil \l^{ {\gamma}} \rceil ,\\
    \sigma &= \lceil \l^{\frac{1}{\gamma}}  \rceil.
    \end{align*}
    
    \item  Temporal concentration $\kappa >0$ and spatial concentration $\mu >0$:
    \begin{align*}
    \kappa &= \l^{ \frac{2}{\gamma} + d+1 - 6 \gamma },\\
    \mu &= \l^{  }.
    \end{align*}
\end{enumerate}
 
For convenience, we insist that $\frac{1}{\gamma
} \in \gamma \NN$ so that if $\nu = \l^{ {\gamma}}$, then $ \sigma = \l^{\frac{1}{\gamma}}$.

Note that we have the hierarchy of parameters
$$
\nu \ll \mu \ll \sigma \ll \kappa^{1/2}.
$$
More precisely, we have the following useful lemma that will be used throughout the next section.

\begin{lemma}\label{lemma:parameters}
For any $\l>0$ such that $\l^{ {\gamma}} \in \NN$, there hold
\begin{align}
\nu \kappa^{\frac{1}{2}} \sigma^{-1} \mu^{-1}\mu^{ \frac{d-1}{2}- \frac{d-1}{r}}&\leq \l^{-\gamma} \label{eq:temporal_constraint}\\
\kappa^{\frac{1}{2}- \frac{1}{p} }     \mu^{  \frac{d-1}{2} } &\leq \l^{-\gamma}\label{eq:LpLinfty_constraint}\\
  \kappa^{- \frac{1}{2} } \sigma \mu \mu^{  \frac{d-1}{2} -\frac{d-1}{q}} &\leq \l^{-\gamma}\label{eq:L1W1q_constraint} .
% \kappa^{-\frac{1+\e}{2(3-\e)}}(\sigma\mu)^{\frac{1}{3}}\mu^{  \frac{d-1}{2} }&\leq\l^{-\gamma}\tag{$L^{\frac{3}{2}-\e}C^{\frac{1}{3}}$}
\end{align}
\end{lemma}
\begin{proof}
The first inequality \eqref{eq:temporal_constraint} is equivalent to
\begin{align*}
\gamma+ \left(\frac{1}{\gamma} + \frac{d+1}{2} - 3\gamma \right) -\frac{1}{\gamma} -1 +\left(\frac{d-1}{2}   - \frac{d-1}{r}\right) \leq -\gamma
\end{align*}
It can be simplified to
$$
d-1 - \frac{d-1}{r} \leq \gamma
$$
which holds due to \eqref{eq:def_of_r}.

For the second one \eqref{eq:LpLinfty_constraint}, thanks to \eqref{eq:value_gamma} it suffices to show that
$$
\kappa^{- d\gamma} \mu^{\frac{d-1}{2}} \leq \l^{-\gamma}.
$$
It is equivalent to 
$$
-2d + \gamma d(d+1-6\gamma) + \frac{d-1}{2} \leq -\gamma
$$
which holds trivially since $10 d\gamma \leq 1$.

The third one \eqref{eq:L1W1q_constraint}, looking again at the exponents, reads as
$$
-\left(\frac{1}{\gamma} + \frac{d  + 1}{2} - 3 \gamma\right) +\frac{1}{\gamma} +1 +\left(\frac{d-1}{2} - \frac{d-1}{q}\right) \leq -\gamma.
$$
Simplifying, we obtain
$$
 - \frac{d-1}{q} \leq - 4 \gamma,
$$
which also holds due to \eqref{eq:value_gamma}.

\end{proof}

In what follows, we can assume without loss of generality that $\sigma = \l^{\frac{1}{\gamma}}$ and $\nu  =   \l^{ {\gamma}}  $ as we will only require $\l$ to be sufficiently large.

\subsection{Estimates on velocity perturbation}

We first estimate the coefficient $a_k$ of the perturbation $w$. Recall the cutoff threshold $\rho$ is defined by \eqref{eq:def_rho} and function $g_{\kappa}$ is nonnegative.
\begin{lemma}\label{lemma:estimates_a_k}
The coefficients $a_k$ are smooth on $[0,1] \times \TT^d$ and  
$$
\| \p_t^n \nabla^m a_k  \|_{L^p(0,1; L^\infty(\TT^d))} \leq C_{u, m,n} (\nu \kappa)^n \kappa^{\frac{1}{2}   - \frac{1}{p}} \quad \text{for any $p\in[1, \infty]$, }
$$
where $C_{\overline{u}, m,n} $ are constants independent of $\nu$ and $\kappa$ (but depending on $\overline{u}$ and hence $\tau$).
In addition, the bound
$$
\|a_k (t) \|_{L^2(\TT^d)} \lesssim \theta(t)  g_{\kappa}(\nu t) \Big( \int_{\TT^d }\rho(t,x) \, dx \Big)^{ \frac{1}{2}}
$$
holds uniformly for all time $t\in[0,1]$.
\end{lemma}
\begin{proof}
It follows from definition \eqref{eq:def_a_k} that $a_k$ is smooth. Since the implicit constant is allowed to depend on $(\overline{u} , \overline{R}) $, it suffices to consider only the time differentiation. 

We have that
\[
\begin{split}
\big\|\p^n_t \big[\theta(\cdot)g_\kappa(  \nu \cdot) \big]\big\|_{L^p([0,1])} &\leq 
\|\theta\|_{C^n([0,1])} \|\p^n_t g_\kappa( \nu \cdot)\|_{L^p([0,1])}  \\
&\lesssim (\nu\kappa)^n \kappa^{\frac{1}{2}   - \frac{1}{p}},
\end{split}
\]
which implies the first bound. The second bound follows immediately from the definition of $a_k$:
$$
\|a_k (t)  \|_{L^2(\TT^d)} \lesssim \theta g_\kappa \left( \int_{\TT^d} \rho \Gamma_k\big(\Id - \frac{ \overline{R}}{\rho }\big)  \,dx   \right)^\frac{1}{2}  \lesssim  \theta g_\kappa \left( \int_{\TT^d} \rho    \,dx   \right)^\frac{1}{2} .
$$

\end{proof}

With the estimates of $a_k$ in hand, we start estimating the velocity perturbation. As expected, the principle part $w^{(p)}$ is the largest among all parts in $w$.
\begin{proposition} \label{p:estimates_on_w^p}
The principle part $w^{(p)}$ satisfies
\begin{align*}
\| w^{(p)} \|_{L^2([0,1] \times \TT^d) }  \lesssim \| \overline{R}  \|_{L^1([0,1] \times \TT^d) }^\frac{1}{2} + C_u \sigma^{-\frac{1}{2}},
\end{align*}
and 
\begin{align*}
\| w^{(p)} \|_{L^p ( 0,1; L^\infty(\TT^d)) } + \| w^{(p)} \|_{L^1 ( 0,1; W^{1,q}(\TT^d)) }  \leq C_u \l^{-\gamma}.
\end{align*}

In particular, for sufficiently large $\l$, 
\begin{align*}
\| w^{(p)} \|_{L^2([0,1] \times \TT^d) } & \lesssim \| \overline{R}  \|^{\frac{1}{2}}_{L^1([0,1] \times \TT^d) }, \\
\| w^{(p)} \|_{L^p( 0,1; L^\infty(\TT^d)) }  & \leq \frac{\delta}{4}.
\end{align*}
\end{proposition}
\begin{proof}
We first show the estimate for $L^2_{t,x}$  and then for $ L^p L^\infty$.

\noindent
{\bf $L^2_{t,x}$ estimate:}

Taking $L^2$ norm in space and appealing to Lemma~\ref{lemma:improved_Holder}, we have
\begin{align*}
\| w^{(p)} (t) \|_{L^2( \TT^d) } \lesssim \sum_{k \in \L } \| a_k(t) \|_2 \| \bwk \|_2 + \sigma^{-\frac{1}{2} } C_u  .
\end{align*}
Recall that $\| \bwk \|_2 \lesssim 1$. Then using Lemma \ref{lemma:estimates_a_k} and taking $L^2$ norm in time gives
\begin{align}\label{eq:estimate_w_p_1}
\| w^{(p)}  \|_{L^2([0,1] \times \TT^d) } \lesssim \sum_{k \in \L } \Big(\int_{0}^1  g_{\kappa}^2(\nu t) \int_{\TT^d} \rho (t,x) \,dx \,dt\Big)^\frac{1}{2} 
+ \sigma^{-\frac{1}{2} } C_u.
\end{align}
Notice that
$$
t \mapsto \int_{\TT^d} \rho (t,x) \,dx 
$$
is a smooth map on $[0,1]$. Thus, we may apply Lemma \ref{lemma:improved_Holder} once again (with $p=1$) to obtain that
\begin{equation}\label{eq:estimate_w_p_2}
\int_{0}^1  g_{\kappa}^2(\nu t) \int_{\TT^d} \rho (t,x) \,dx \,dt \lesssim \| \overline{R} \|_{L^1([0,1] \times \TT^d) } +
  C_u   \nu^{-1}   ,
\end{equation}
where we have used the fact that $\int g_\kappa^2 =1$ and thanks to \eqref{eq:def_rho} the bound 
$$
\int_{\TT^d} \rho (t,x) \,dx \lesssim \|\overline{R} (t)\|_{L^1(\TT^d)} + \| \overline{R}\|_{L^1 {([0,1]\times \TT^d)}}.
$$

Hence, combining \eqref{eq:estimate_w_p_1} and \eqref{eq:estimate_w_p_2} gives
\begin{align*}
\| w^{(p)}  \|_{L^2([0,1] \times \TT^d) } 
&\lesssim  \| \overline{R} \|_{L^1([0,1] \times \TT^d) }^\frac{1}{2}  + C_u  \sigma^{-\frac{1}{2} } . 
\end{align*}

\noindent
{\bf$ L^p_t L^\infty$ estimate:}

Taking $L^\infty$ norm in space and using H\"older's inequality give
\begin{align*}
\| w^{(p)} (t) \|_{L^\infty( \TT^d) } \lesssim \sum_{k \in \L } \| a_k(t) \|_\infty \| \bwk \|_\infty.
\end{align*}
We can now take $L^p$ in time and apply the estimates in Lemma~\ref{lemma:estimates_a_k} and Theorem~\ref{thm:main_thm_for_W_k}  to obtain
\begin{align*}
\| w^{(p)}   \|_{L^p(0,1;L^\infty( \TT^d)) } & \lesssim \mu^{  \frac{d-1}{2}  } \sum_{k \in \L } \| a_k  \|_{L^p(0,1;L^\infty( \TT^d)) }       \\
&\leq C_u \kappa^{ \frac{1}{2} -\frac{1}{p} }    \mu^{  \frac{d-1}{2}  },
\end{align*}
which by Lemma \ref{lemma:parameters} implies that
\begin{align*}
\| w^{(p)}   \|_{L^p(0,1;L^\infty( \TT^d)) }  & \leq C_u   \l^{ -  \gamma } .
\end{align*}

\noindent
{\bf$ L^1_t W^{1,q}$ estimate:}

This part is similar to the $L^p L^\infty$. We first take $W^{1,q}$ norm in space  to obtain
\begin{align*}
\| w^{(p)} (t) \|_{W^{1,q}( \TT^d) } \lesssim \sum_{k \in \L } \| a_k(t) \|_{C^1} \| \bwk ( \sigma \cdot )  \|_{W^{1,q}}.
\end{align*}

Integrating in time, by Theorem \ref{thm:main_thm_for_W_k}, Lemma \ref{lemma:estimates_a_k} we have that
\begin{align*}
\| w^{(p)}  \|_{L^1(0,1;W^{1,q}( \TT^d) )} & \lesssim \sum_{k \in \L } \| a_k \|_{L^1_t C^1} \| \bwk ( \sigma \cdot ) \|_{W^{1,q}} \\
& \lesssim \sigma \mu \mu^{ \frac{d-1}{2}  -\frac{d-1}{q} }\sum_{k \in \L } \| a_k(t) \|_{L^1_t  C^1}  \\
&\lesssim \sigma \kappa^{- \frac{1}{2}}\mu \mu^{ \frac{d-1}{2}  -\frac{d-1}{q} } ,
\end{align*}
which implies the desired bound thanks to Lemma \ref{lemma:parameters}.

\end{proof}

Next, we estimate the corrector $ w^{(c)}$, which is expected to be much smaller than $w^{(p)}$ due to the derivative gains from both the fast oscillation $\sigma $ and the tensor potential $\bf\Omega_k$ defined in \eqref{eq:def_bf_O_k}.
\begin{proposition} \label{p:estiates_on_w^c}
The divergence-free corrector $w^{(c)}$ satisfies
\begin{align*}
\| w^{(c)} \|_{L^2(0,1; L^\infty  (\TT^d) ) }  \leq   C_u   \lambda^{- \gamma },
\end{align*}
 and
\begin{align*}
\| w^{(c)} \|_{L^1(0,1; W^{1,q}  (\TT^d) ) }  \leq   C_u   \lambda^{- \gamma }   .
\end{align*}

In particular, for sufficiently large $\l$, 
\[
\| w^{(c)} \|_{L^2([0,1] \times \TT^d) } \leq \| \overline{R}  \|^{\frac{1}{2}}_{L^1([0,1] \times \TT^d) },
\]
and
\[
\| w^{(c)} \|_{L^p( 0,1; L^\infty (\TT^d)) } +\| w^{(c)} \|_{L^1(0,1; W^{1,q}  (\TT^d) ) }   \leq \frac{\delta}{4}.
\]
\end{proposition}
\begin{proof}

%%%%%%%%%%%%%%%%%%%%%%%%%%%%%
\noindent
{\bf$ L^2_t L^{\infty}$ estimate:}

From the definition, we have
\[
\begin{split}
\|w^{(c)} (t)\|_{L^\infty  (\TT^d)} &\leq   \sigma^{-1}\Big\| \sum_{k \in \L }   \nabla  a_k(t  ) : \bok(\sigma \cdot )   \Big\|_{L^\infty  (\TT^d)}\\
&\lesssim  \sigma^{-1} \sum_{k \in \L } \|\nabla a_k (t)\|_{L^\infty  (\TT^d)} \| \bok(\sigma \cdot )  \|_{L^\infty  (\TT^d)}.
\end{split}
\]
Now, thanks to Theorem~\ref{thm:main_thm_for_W_k}, Lemma~\ref{lemma:estimates_a_k}, and \eqref{eq:def_of_r}, we take $L^2$ in time to obtain
\[
\begin{split}
\| w^{(c)} \|_{L^2(0,1; L^{\infty  }  (\TT^d) ) }  &\lesssim   \sigma^{-1} \mu^{-1 + \frac{d-1}{2}} \sum_{k \in \L } \|\nabla a_k    \|_{L^2(0,1; L^{\infty  } (\TT^d)) }  \\
& \leq C_u   \sigma^{-1} \mu^{-1 + \frac{d-1}{2}}  ,
\end{split}
\]
which by the definition of $\gamma$ implies that
$$
\| w^{(c)} \|_{L^2(0,1; L^{\infty  }  (\TT^d) ) }  \leq C_u \l^{-\gamma}.
$$

\noindent
{\bf$ L^1_t W^{1,q}$ estimate:}
This part is very similar to the estimation of $w^{(p)}$. We first take $W^{1,q}$ in space to obtain that
\begin{align*}
\|w^{(c)} (t)\|_{ W^{1,q}  (\TT^d)} &\leq   \sigma^{-1}\Big\| \sum_{k \in \L }   \nabla  a_k(t  ) : \bok(\sigma \cdot )   \Big\|_{ W^{1,q}  (\TT^d)}\\
&\lesssim  \sigma^{-1} \sum_{k \in \L } \| a_k (t)\|_{C^2  (\TT^d)} \| \bok(\sigma \cdot )  \|_{W^{1,q} (\TT^d)}.
\end{align*}

Integrating in space and using Lemma \ref{lemma:estimates_a_k} and Theorem \ref{thm:main_thm_for_W_k} we have
\begin{align*}
\|w^{(c)}  \|_{L^1(0,1; W^{1,q}  (\TT^d))}  
&\lesssim  \sigma^{-1} \sum_{k \in \L } \| a_k \|_{L^1(0,1; C^2  (\TT^d)  )} \| \bok(\sigma \cdot )  \|_{W^{1,q} (\TT^d)}\\
 &\lesssim \kappa^{-\frac{1}{2}} \mu^{ \frac{d-1}{2}  - \frac{d-1}{q}},
\end{align*}
which differs the estimate of $\|w^{(p)}  \|_{L^1_t   W^{1,q}   }   $ by a factor of $\sigma \mu$ and hence
\begin{align*}
\|w^{(c)}  \|_{L^1(0,1; W^{1,q}  (\TT^d))}  
&\lesssim    C_u \l^{-\gamma}.
\end{align*}

\end{proof}

Finally, we estimate the temporal corrector $w^{(t)}$. From its definition \eqref{eq:def_w_t}, one can see that the spatial frequency of $w^{(t)}$ is independent from the parameters $\sigma$, $\tau$ and $\mu$. As a result, this term poses no constraints to the choice of temporal and spatial oscillation/concentration at all and is small for basically any choice of parameters (as long as temporal oscillation $\nu$ is present). This is one of the main technical differences from \cite{MR3898708,1809.00600} where the leading order effect is temporal oscillation.

\begin{proposition} \label{prop:estimates_on_w^t}
The temporal corrector $w^{(t)}$ satisfies
\begin{align*}
\| w^{(t)} \|_{L^\infty( 0,1; W^{1,\infty} ( \TT^d) )}  \leq  C_u \nu^{-1}.
\end{align*}

In particular, for sufficiently large $\l$, 
\[
\| w^{(t)} \|_{L^2([0,1] \times \TT^d) } \leq \| \overline{R}  \|^{\frac{1}{2}}_{L^1([0,1] \times \TT^d) },
\]
and
\[
\| w^{(t)} \|_{L^p( 0,1; L^\infty(\TT^d)) } + \| w^{(t)} \|_{L^1( 0,1; W^{1,q}(\TT^d)) }  \leq \frac{\delta}{4}.
\]
\end{proposition}
\begin{proof}

It follows directly from the definition of $w^{(t)}$ that
$$
\| w^{(t)} \|_{L^\infty( 0,1; W^{1,\infty} ( \TT^d) )}   \lesssim    \nu^{-1} \| h\|_{L^\infty([0,1])} \|  \overline{R}    \|_{L^\infty( 0,1; W^{2,\infty} ( \TT^d) )}  \leq \nu^{-1} C_u,
$$
where in the last step we have used \eqref{eq:bound_on_h_k}.
\end{proof}

\subsection{Estimates on the new Reynolds stress}
The last step of the proof is to estimate $R_1$. We proceed with the decomposition in Lemma \ref{lemma:new_stress_R_1}. More specifically, we will prove that for all sufficiently large $\l$, each part of the stress $R_1$ is less than $\frac{\delta}{4} $.

\subsubsection{Linear error}
\begin{lemma}
For sufficiently large $\lambda$,
\[
\| R_{\Lin} \|_{L^1(0,1; L^r(\TT^d))} \leq \frac{\delta}{4}.
\]
\end{lemma}
\begin{proof}
We split the linear error into three parts:
\begin{equation*} 
\| R_{\Lin}\|_{L^1(0,1; L^r(\TT^d))} \leq   \| \underbrace{\mathcal{R}\left(\Delta w \right)   \|_{L^1_t  L^r }}_{:=L_1} +     \underbrace{  \| \mathcal{R}\big( \p_t (w^{(p)}+ w^{(c)})  \big)  \|_{L^1_t  L^r } }_{:= L_2 } +     \underbrace{ \|   \mathcal{R}\left( \D ( w \otimes \overline{u} )+ \D (\overline{u} \otimes w ) \right)  \|_{L^1_t  L^r } }_{:=L_3}.
\end{equation*}

%%%%%%%%%%%%%%%%%%%%%%%%%%%%%%%%%%%%%%%%%%%%%%
\noindent
{\bf Estimate of  $L_1$:}

By \eqref{eq:appendix_R_2} or boundedness of Riesz transform we have
\begin{align*}
 L_1   &\lesssim   \|   w  \|_{L^1(0,1; W^{1,r}(\TT^d))} .
 \end{align*}
Note that we have estimated $w $ in $L^1_t  W^{1,q}$ and $r<2<  q$. Thus by Proposition \ref{p:estimates_on_w^p}--\ref{prop:estimates_on_w^t} we can conclude that
\begin{equation}\label{eq:proof_new_stress_lin_2}
L_1 \leq C_u \l^{- \gamma}.
\end{equation}

%%%%%%%%%%%%%%%%%%%%%%%%%%%%%%%%%%%%%%%%%%%%%%
\noindent
{\bf Estimate of  $L_2$:}

By \eqref{eq:w_p_plus_w_c}, we have
$$
\p_t (w^{(p)}+ w^{(c)}) = \sigma^{-1} \sum_k \D( \p_t a_k \bok (\sigma \cdot )),
$$
and hence
\begin{align*}
L_2 & \leq \|\mathcal{R} \p_t (w^{(p)}+ w^{(c)}) \|_{L^1(0,1; L^r(\TT^d))} \\
 & \lesssim \sigma^{-1}\sum_k \| \mathcal{R}\D(\p_t a_k \bok  (\sigma \cdot )  ) \|_{L^1(0,1; L^r(\TT^d))}  .
 \end{align*}
 Since $ \mathcal{R}\D$ is a Calder\'on-Zygmund operator on $\TT^d$, we have
\begin{align*}
L_2 & \lesssim \sigma^{-1} \sum_k \| \p_t a_k \|_{L^1_t L^\infty }   \|\bok    \|_{r}   .
\end{align*}

Appealing to Lemma \ref{lemma:estimates_a_k} and estimates of $\bok  $ listed in Theorem \ref{thm:main_thm_for_W_k} we have
\begin{align}
L_2 & \leq C_{u} \sigma^{-1} (\nu \kappa) \kappa^{-\frac{1}{2}}   \mu^{-1+\frac{d-1}{2}  - \frac{d-1}{r}} \nonumber \\
& \leq C_u  \l^{-\gamma},\label{eq:proof_new_stress_lin_3}
\end{align}
where we used \eqref{eq:temporal_constraint} for the second inequality.

%%%%%%%%%%%%%%%%%%%%%%%%%%%%%%%%%%%%%%%%%%%%%%
\noindent
{\bf Estimate of  $L_3$:}
For the last term we simply use $ L^r$ boundedness of $\mathcal{R} \D$ to obtain
\begin{align*}
L_3 \lesssim \|      w \otimes \overline{u} \|_{L^1_t  L^r } + \|       \overline{u} \otimes w \|_{L^1_t  L^r }.   
\end{align*}
Here we use a crude bound
\begin{align*}
 \|      w \otimes \overline{u} \|_{L^1_t  L^r } + \|       \overline{u} \otimes w \|_{L^1_t  L^r } \lesssim \| w \|_{L^p_t  L^\infty  }  \| \overline{u} \|_{L^\infty_{t,x}  } , 
\end{align*}
and apply the obtained estimates in Proposition \ref{p:estimates_on_w^p}, \ref{p:estiates_on_w^c} and \ref{prop:estimates_on_w^t} to conclude
\begin{equation}\label{eq:proof_new_stress_lin_4}
    L_3 \leq C_u \l^{- \gamma }.
\end{equation}

From \eqref{eq:proof_new_stress_lin_2}, \eqref{eq:proof_new_stress_lin_3}, and \eqref{eq:proof_new_stress_lin_4}, we can conclude that for all sufficiently large $\l$, there holds
\begin{align*}
\| R_{\Lin} \|_{L^1(0,1; L^r(\TT^d))} \leq  \frac{\delta}{4}.
\end{align*}
\end{proof}

\subsubsection{Correction error}
\begin{lemma}
For sufficiently large $\lambda$,
\[
\| R_{\Cor} \|_{L^1(0,1; L^r(\TT^d))} \leq \frac{\delta}{4}.
\]
\end{lemma}
\begin{proof}
 
By the boundedness of $\mathcal{R}\D$ in $L^r$, $1<r<2$, and H\"older's inequality,
\begin{equation*}
\begin{split}
\|R_{\Cor}\|_{L^1(0,1; L^r(\TT^d))} &\lesssim
\|(w^{(c)}+w^{(t)})  \otimes w\|_{L^1_t L^r }   + \|w^{(p)} \otimes ( w^{(c)} + w^{(t)} )\|_{L^1_t L^r }\\
& \lesssim 
\left(\|w^{(c)}\|_{L^2_t L^\infty }+\|w^{(t)}\|_{L^2_t L^\infty} \right) \|w\|_{L^2_{t,x} }   + \|w^{(p)}\|_{L^2_{t,x} } \left( \|w^{(c)}\|_{L^2_t L^{\infty} } + \|w^{(t)}\|_{L^2_t L^{\infty} } )\right).
\end{split}
\end{equation*}

By Propositions~\ref{p:estimates_on_w^p}, \ref{p:estiates_on_w^c}, and \ref{prop:estimates_on_w^t},
\[
\begin{split}
\|w\|_{L^2_{t,x}  } &\leq \|w^{(p)}\|_{L^2_{t,x }} + \|w^{(c)}\|_{L^2_{t,x}} +\|w^{(t)}\|_{L^2_{t,x}}\\
&\lesssim \| \overline{R}  \|^{\frac{1}{2}}_{L^1([0,1] \times \TT^d) },
\end{split}
\]
and  
\[
\|w^{(c)}\|_{L^2_t L^{ \infty } }+\|w^{(t)}\|_{L^2_t L^{\infty } } \leq C_u  \lambda^{- \gamma  }  .
\]
So for all $\lambda$ sufficiently large, we can conclude that
\[
\|R_{\Cor}\|_{L^1(0,1; L^r(\TT^d))} \leq \frac{\delta}{4}.
\]

\end{proof}

\subsubsection{Oscillation error}

\begin{lemma}
For sufficiently large $\lambda$,
\[
\| R_{\Osc} \|_{L^1(0,1; L^r(\TT^d))} \leq \frac{\delta}{4}.
\]
\end{lemma}
\begin{proof}
We will use the decomposition from Lemma~\ref{lemma:R_osc_decomposition}
$$
R_{\Osc} = R_{\Osc,x}   + R_{\Osc,t} + R_{\Far}.
$$

%%%%%%%%%%%%%%%%%%%%%%%%%%%%%%%%%%%%%%%%%%%%%%
\noindent
{\bf Estimate of  $R_{\Osc,x}$:}

Denote $\mathbf{T}_k :[0,1] \times \TT^d \to \RR^{d\times d}$ by 
$$
\mathbf{T}_k =  \bwk   \otimes \bwk   - \fint \bwk \otimes \bwk  , 
$$
so that
$$
R_{\Osc,x} =  \sum_{k \in \L} \mathcal{B}\big(\nabla(a_k^2 ) , \mathbf{T}_k(\sigma \cdot) \big) .
$$

Using Theorem \ref{thm:bounded_B} and the fact that $\mathbf{T}_k $ has zero spatial mean, we can estimate the $L^r$ norm of $R_{\Osc,x}$ as follows.  
\[
\begin{split}
\|R_{\Osc,x}(t) \|_{L^r(\TT^d)} &=  \Big\|   \sum_{k \in \L} \mathcal{B}\big(  \nabla(a_k^2 ) , \mathbf{T}_k(\sigma \cdot)  \big) \Big\|_{L^r}\\
&\lesssim \sum_{k \in \L} \|\nabla (a_k^2)\|_{C^1} \| \mathcal{R}\big(  \mathbf{T}_k(\sigma \cdot)  \big)\|_{L^r}\\
&\lesssim \sigma^{-1}\sum_{k \in \L} \|\nabla (a_k^2)\|_{C^1} \|\mathbf{T}_k  \|_{L^r},
\end{split}
\]
where the last inequality used the fact that $ \mathbf{T}_k(\sigma \cdot)$ has zero spatial mean.
Thanks to Theorem~\ref{thm:main_thm_for_W_k}, for any $k\in \L$
\[
\| \mathbf{T}_k\|_{L^r} \lesssim \| \bwk \otimes \bwk   \|_{L^r} \lesssim \|\bwk\|_{L^{2r}}^2 \lesssim  \mu^{d- 1-\frac{d -1}{r}}.
\]
Therefore, by Lemma~\ref{lemma:estimates_a_k},
\[
\begin{split}
\| R_{\Osc,x} \|_{L^1(0,1; L^r(\TT^d))} &\leq C_u \sigma^{-1}   \mu^{d- 1-\frac{d-1}{r}} .
\end{split}
\]

%%%%%%%%%%%%%%%%%%%%%%%%%%%%%%%%%%%%%%%%%%%%%%%%%%%%%%

%%%%%%%%%%%%%%%%%%%%%%%%%%%%%%%%%%%%%%%%%%%%%%%%%%

\noindent
{\bf Estimate of $R_{\Osc,t}$:} 

Using the bound on $h_k$ \eqref{eq:bound_on_h_k}, we infer
\[
\begin{split}
\|R_{\Osc,t}\|_{L^1(0,1; L^r(\TT^d))} &= \|\nu^{-1} h_k(\nu t)  \D \p_t \overline{R}\|_{L^1L^r}\\
&\lesssim  \nu^{-1} \|h_k(\nu \cdot)\|_{L^1} C_u\\
&\leq  C_u\nu^{-1}  .
\end{split}
\]

%%%%%%%%%%%%%%%%%%%%%%%%%%%%%%%%%%%%%%%%%%%%%
\noindent
{\bf Estimate of $R_{\Far}$:} 

We can use Theorem~\ref{thm:main_thm_for_W_k} and Lemma~\ref{lemma:estimates_a_k} to obtain
\begin{align*}
\|R_{\Far }\|_{L^1(0,1; L^r(\TT^d))} & = \Big\|  \sum_{k \neq k'} a_k a_{k'} \bwk (\sigma \cdot) \otimes \bw_{{k'}}(\sigma \cdot)\Big\|_{L^1(0,1; L^r(\TT^d))}\\
&\lesssim \sum_{k \neq k'}  \|a_k\|_{L^2(0,1; L^{\infty}(\TT^d))}  \|a_{k'}\|_{L^2(0,1; L^{\infty}(\TT^d))}\|\bwk \otimes  \bw_{{k'}} \|_{L^{r}} \\
&\leq C_u \mu^{d-1 - \frac{d}{r}}.
\end{align*}

Now we can combine all the estimates and conclude
\[
\begin{split}
\| R_{\Osc} \|_{L^1(0,1; L^r(\TT^d))} & \leq C_u \left( \sigma^{-1}   \mu^{d- 1-\frac{d-1}{r}}+  \nu^{-1}   + \mu^{d-1 - \frac{d}{r}}  \right).
\end{split}
\]
Thanks to \eqref{eq:def_of_r}, we have
\[
\| R_{\Osc} \|_{L^1(0,1; L^r(\TT^d))} \leq C_u   \l^{-  \gamma}   .
\]
And thus for $\lambda$ large enough, the desired bound holds:
\[
\| R_{\Osc} \|_{L^1(0,1; L^r(\TT^d))} \leq \frac{\delta}{4}.
\]

\end{proof}

\appendix

\section{$X^{p,q}$ weak solutions on the torus}\label{sec:append_weak}

In this section, we show that sub-critical and critical weak solutions in $ X^{p,q}([0,T] ; \TT^d)$ are in fact Leray-Hopf. In particular, by the weak-strong uniqueness of Ladyzhenskaya-Prodi-Serrin, this implies the uniqueness part of Theorem \ref{thm:FJR_uniqueness}. The content of Theorem \ref{thm:critical_Leray-Hopf} is classical~\cite{MR316915,MR760047,MR1813331,MR1876415} and we include a proof in the regime $q>2$ for the convenience of the readers. Note that the proof applies to the case $q=\infty$ which is most relevant to the results of this paper, but was omitted in \cite{MR316915}.

\begin{theorem}\label{thm:critical_Leray-Hopf}
Let $d \geq 2$ be the dimension and $u \in X^{p,q}([0,T] ; \TT^d)$ be a weak solution of \eqref{eq:NSE} with $\frac{2}{p} + \frac{d}{q} = 1$, $d \leq q \leq \infty$. Then $u$ is a Leray-Hopf solution.
\end{theorem}

We prove Theorem \ref{thm:critical_Leray-Hopf} for $q>2$ only. The case $d=q=2$,  as discussed in
\cite{MR1891170}, can be handled by the argument of \cite{MR1813331} in 2D. The method we present here follows the duality approach in \cite{MR1876415} and use only classical ingredients.

The first ingredient is an existence result for a linearized Navier-Stokes equation.
\begin{theorem}\label{thm:Leray-Hopf_v}
Let $u \in X^{p,q}([0,T] ; \TT^d)$ be a weak solution of \eqref{eq:NSE} with $\frac{2}{p} + \frac{d}{q} = 1$. For any divergence-free $v_0 \in L^2(\TT^d)$, there exists a weak solution $v \in C_w L^2 \cap L^2_t  H^1$ to the linearized Navier-Stokes equation:
\begin{equation}\label{eq:eq_for_v}
\begin{cases}
\p_t v  -\Delta v + u \cdot \nabla v + \nabla p = 0 &\\
\D v =0,
\end{cases}
\end{equation}
satisfying the energy inequality
\[
\frac{1}{2}\|v(t)\|_2^2+ \int_{t_0}^t \|\nabla v(s)\|_2^2 \, ds\leq \frac{1}{2}\|v(t_0)\|_2^2 , 
\]
for all $t\in [t_0,T]$, a.e. $t_0\in[0,T]$ (including $t_0=0$).
\end{theorem}
\begin{proof}
This follows by a standard Galerkin method and can be found in many textbooks. See~\cite[Chapter 4]{MR3616490} or \cite{MR1284206} for details.
\end{proof}

Let $v$ be the weak solution given by Theorem \ref{thm:Leray-Hopf_v} with initial data $u(0)$. The goal is to show $u \equiv v$. Setting $w = u -v$, the equation for $w$ reads
\begin{equation*}
\p_t w  -\Delta w + u \cdot \nabla w + \nabla q = 0 ,
\end{equation*}
and its weak formulation
\begin{equation}\label{eq:difference_w_weak_formulation}
  \int_0^T \int_{\TT^d} w\cdot ( \partial_t \varphi + \Delta \varphi + u\cdot  \nabla   \varphi     ) \, dx dt =0 \quad \text{for any $\varphi \in \mathcal{D}_T$},
\end{equation}
where we recall that the test function class $\mathcal{D}_T $ consists of smooth divergence-free functions vanishing for $t \geq T$.

Fix $F \in C^\infty_c([0,T] \times \TT^d)$. Let $\Phi:[0,T]\times \TT^d \to \RR^d$ and $\chi:[0,T]\times \TT^d \to \RR  $ satisfy the system of equations
\begin{equation}\label{eq:dual_problem_Phi}
\begin{cases}
-\p_t \Phi -\Delta \Phi - u \cdot \nabla \Phi + \nabla \chi = F&\\
\D \Phi =0 &\\
\Phi (T) =0.
\end{cases}
\end{equation}
Note that the equation of $ \Phi$ is ``backwards in time'' and  by a change of variable one can convert \eqref{eq:dual_problem_Phi} into a more conventional form.

If we can use $\varphi=\Phi$ as the test function in the weak formulation \eqref{eq:difference_w_weak_formulation}, then immediately
$$
\int_{[0,T] \times \TT^d} w \cdot F \, dx dt = 0 .
$$
Since $F \in C^\infty_c([0,T] \times \TT^d)$ is  arbitrary, we have 
$$
w = 0 \quad \text{for $a.e.$ $(t,x) \in [0,T] \times \TT^d$}.
$$

So the question of whether $u \equiv v$ reduces to showing a certain regularity of $\Phi$. More specifically, we can prove the following theorem.

\begin{theorem}\label{thm:appendix_reg_Phi}
Let $d \geq 2$ be the dimension and $u \in X^{p,q}([0,T] ; \TT^d)$ be a weak solution of \eqref{eq:NSE} with $\frac{2}{p} + \frac{d}{q} = 1$ with $q>2$. For any $F  \in C^\infty_c([0,T] \times \TT^d) $, the system \eqref{eq:dual_problem_Phi} has a weak solution $ \Phi \in L^\infty_t L^2 \cap L^2_t H^1$ such that $\Phi$ can be used as a test function in \eqref{eq:difference_w_weak_formulation}.

\end{theorem}
\begin{proof}
We will prove the weak solution $\Phi$ satisfies the regularity 
\begin{equation}\label{eq:appendix_aux1}
\p_t \Phi ,\Delta\Phi ,u\cdot\nabla \Phi, \nabla \chi   \in
 L^2_{t,x}  ,
\end{equation}
which implies $ \Phi$ can be used in \eqref{eq:difference_w_weak_formulation} since $w \in L^2_{t,x}$.

The solution $ \Phi$ will be constructed by the Galerkin method of the following finite-dimensional approximation 
\begin{equation} \label{Galerkin}
\begin{cases}
-\p_t \Phi_n -\Delta \Phi_n - P_n\big[ u_n \cdot \nabla \Phi_n\big] + \nabla \chi_n = F_n &\\
\D \Phi_n =0 &\\
\Phi_n (T) =0,
\end{cases}
\end{equation}
where $\Phi_n,  u_n, \chi_n, F_n$ are restricted to the first $n$ Fourier modes, and $P_n$ is the projection operator on those modes.

It suffices to verify the following \emph{a priori} estimates as they are preserved in the limit as $n \to \infty$.  We also only need to show the estimates for $\p_t \Phi$, $ \Delta\Phi $,   and $u\cdot\nabla \Phi$ since the pressure $  \chi $ satisfies the equation
$$
\Delta \chi = \D ( u \cdot  \nabla \Phi)+ \D F.
$$

\noindent
{\bf Step 1:} Energy bounds $ L^\infty_t L^2 \cap L^2_t H^1 $.

The solutions to the Galerkin approximation \eqref{Galerkin} enjoy the energy estimate
$$
-\frac{1}{2} \frac{d}{dt} \| \Phi \|_2^2 + \| \nabla \Phi \|_2^2 \leq \int_{\TT^d} |F \cdot \Phi|\,dx,
$$
which implies the desired energy bounds.

\noindent
{\bf Step 2:} Higher bounds $ L^\infty_t H^1 \cap L^2_t H^2 $.

We now take the $L^2$ inner product of \eqref{Galerkin} with $\Delta \Phi$ to obtain
$$
-\frac{1}{2} \frac{d}{dt} \| \nabla \Phi \|_2^2 + \| \Delta \Phi \|_2^2 \leq \int_{\TT^d} |F \cdot \Delta \Phi|\,dx  + \int_{\TT^d} |u \cdot \nabla \Phi  \cdot \Delta \Phi |\,dx.
$$

\begin{itemize}
    \item {\bf Case 1:} $ p<\infty$ and $d<q\leq \infty$.

Thanks to Proposition \ref{prop:appendix_trilinear_term_estiamte}, 
\begin{equation}\label{eq:appendix_nonlinear}
\int_{\TT^d} |u \cdot \nabla \Phi  \cdot \Delta \Phi |\,dx \lesssim   \|u  \|_q \| \nabla \Phi \|_2^{ \frac{q-d}{q}} \|  \Delta \Phi \|_2^{ \frac{q+d}{q}}.  
\end{equation}
% When $q=\infty$, by H\"older's inequality we have
% \[
% \int_{\TT^d} |u \cdot \nabla \Phi  \cdot \Delta \Phi |\,dx \lesssim   \|u  \|_q \| \nabla \Phi \|_2 \|  \Delta \Phi \|_2. 
% \]
Then H\"older's, Poincar\'e's, and Young's inequalities yield
\[
-\frac{1}{2} \frac{d}{dt} \| \nabla \Phi \|_2^2 + \frac{1}{2}\| \Delta \Phi \|_2^2 \lesssim \|u\|_q^{\frac{2q}{q-d}}\|\nabla \Phi\|_2^2 + \|F\|_2^2.
\]
Thanks to the integrability of $\|u\|_q^{\frac{2q}{q-d}}$ ($\|u\|_\infty^2$ when $q=\infty$),  Gronwall's inequality immediately implies that $\Phi\in L^\infty_t H^1 \cap L^2_t H^2$.

\item  {\bf Case 2:} $ p=\infty$ and $q=d > 2$.

Since $u \in C_t L^{d}$, for any $\ep>0$ there exists a decomposition $u = u_1  +u_2$ such that
$$
\| u_1 \|_{C_t L^{d} } \leq \ep\quad \text{ and }\quad  u_2 \in L^\infty_{t,x},
$$
and then the $u_1$ portion of the nonlinear term in \eqref{eq:appendix_nonlinear} can be absorbed by $\| \Delta \Phi \|_2^2$, so we arrive at the same conclusion.

\end{itemize}

% \noindent
% {\bf Step 2 continued:} Higher bounds $ L^\infty_t H^{\frac{1}{2}} \cap L^2_t H^{\frac{3}{2}} $ when $q =d=2 $.

% To fix the lack of embedding when $d=q =2$, we take the $L^2$ inner product of \eqref{Galerkin} with $ \L \Phi$, where $\L^s := (-\Delta)^\frac{s}{2}$ is the fractional Laplacian on $\TT^2$ for zero mean functions, and obtain
% $$
% -\frac{1}{2} \frac{d}{dt} \| \L^\frac{1}{2} \Phi \|_2^2 + \| \L^{\frac{3}{2} } \Phi \|_2^2 \leq \int_{\TT^2} |F \cdot \L \Phi  |\,dx  + \int_{\TT^2} |u \cdot \nabla \Phi  \cdot \L \Phi |\,dx.
% $$
% Then the Sobolev embedding $  H^\frac{1}{2}(\TT^2) \hookrightarrow L^4(\TT^2) $ gives
% \begin{equation}\label{eq:appendix_q_2}
% \int_{\TT^2} |u \cdot \nabla \Phi  \cdot \L \Phi |\,dx \lesssim \| u\|_2 \| \nabla \Phi \|_4^2 \lesssim \| u\|_2 \| \L^\frac{3}{2} \Phi \|_2^2.
% \end{equation}
% With \eqref{eq:appendix_q_2} in hand, the above decomposition $u=u_1+u_2$ and Gronwall's inequality imply that $\Phi \in L^\infty_t H^{\frac{1}{2}} \cap L^2_t H^{\frac{3}{2}}$.

\noindent
{\bf Step 3:} Conclusion from maximal regularity of the heat equation.

Taking Leary's projection $\mathbb{P}$ onto the divergence-free vector fields, the equation for $\Phi$ can be rewritten as
\begin{equation*}
\begin{cases}
-\p_t \Phi -\Delta \Phi  = \mathbb{P} (u \cdot \nabla \Phi) +  \mathbb{P} F&\\
\Phi (T) =0.
\end{cases}
\end{equation*}
Therefore, by the maximal $L^p_t L^q$ regularity of the heat equation  (see for instance~\cite[Theorem 5.4]{MR3616490}), we only need to show the estimates \eqref{eq:appendix_aux1} for $u \cdot  \nabla  \Phi$ to conclude the proof.
By Sobolev interpolations we have $\nabla \Phi \in L^r L^s $ for any $\frac{2}{r} + \frac{d}{s} =\frac{d}{2} $ such that $ 2\leq s \leq \frac{2d}{d-2}$. We can find $r,s$ in this regime that satisfy the H\"older relations $ \frac{1}{p} + \frac{1}{r}  = \frac{1}{2}$ and $ \frac{1}{q} + \frac{1}{s}  = \frac{1}{2}$, where recall $\frac{2}{p} + \frac{d}{q} =1$. Since $ u \in L^{p}_t L^q$ with some $\frac{2}{p} + \frac{d}{q} =1$, this choice of $r,s$ implies that
$$
u \cdot  \nabla  \Phi \in L^2([0,T] \times \TT^d ) .
$$

\end{proof}

The last result is a classical estimate of the nonlinear term, cf. \cite[pp. 172]{MR3616490}. Notice that one of the embedding fails when $q=2$ which is the reason we can only prove Theorem \ref{thm:appendix_reg_Phi} for $q > 2$.

\begin{proposition}\label{prop:appendix_trilinear_term_estiamte}
Let $ d \geq 2$ be the dimension and $d   \leq q \leq \infty $ such that $q > 2$. For any smooth vector fields $u,v \in C^\infty_0(\TT^d)$,
$$
\int_{\TT^d} |u \cdot \nabla v \cdot  \Delta v  | \,dx \lesssim  \| u \|_{q}  \| \nabla  v \|_{2}^{\frac{q-d}{q}}      \| \Delta  v \|_{2}^{\frac{q+d}{q}}.
$$

\end{proposition}
\begin{proof}
We apply H\"older's inequality with exponents $\frac{1}{q} + \frac{1}{r} + \frac{1}{2} =1$, $r = \frac{2q}{q-2} \in [2 , \frac{2d}{d-2})$, 
\begin{align*}
\int_{\TT^d} |u \cdot \nabla v \cdot   \Delta v  | \,dx \lesssim \| u \|_{q}  \| \nabla  v \|_{r}        \|\Delta  v \|_{ 2} .
\end{align*}
Since $ 2 \leq r< \infty$, by the  Sobolev embedding  $H^{s}(\TT^d )  \hookrightarrow   L^{ r} (\TT^d )$, $s = d( \frac{1}{2} - \frac{1}{r})$,
\begin{align*}
  \| \nabla  v \|_{r}   \lesssim \|\nabla  v \|_{H^{s}}  .
\end{align*}
Finally, by a standard Sobolev interpolation and the $L^2$-boundedness of Riesz transform, we have
\begin{align*}
 \|\nabla  v \|_{H^{s}}    \lesssim   \| \nabla v \|_{2}^{\frac{q- d}{q}}\| \Delta   v \|_{2}^{\frac{d}{q} },
\end{align*}
which concludes the proof.
\end{proof}

\section{Some technical tools}\label{sec:append_tech}
\subsection{Improved H\"older's inequality on $\TT^d$}
We recall the following result due to Modena and Sz\'ekelyhidi \cite{MR3884855}, which was inspired by \cite[Lemma 3.7]{MR3898708}. This lemma allows us to quantify the decorrelation in the usual H\"older's inequality when we increase the oscillation of one function.

\begin{lemma}\label{lemma:improved_Holder}
Let $p \in [1,\infty]$ and $a,f :\TT^d \to \RR$ be smooth functions. Then for any $\sigma \in \NN$  ,
\begin{equation}
\Big|   \|a f(\sigma \cdot ) \|_{p }  - \|a \|_{p} \| f \|_{p } \Big|\lesssim \sigma^{-\frac{1}{p}} \| a\|_{C^1} \| f \|_{ p }.
\end{equation}
\end{lemma}

The proof is based on the interplay between the Poincare's inequality and the fast oscillation of $f (\sigma \cdot )$ and can be found in \cite[Lemma 2.1]{MR3884855}.

\subsection{Tensor-valued antidivergence $\mathcal{R}$}

For any $f\in C^\infty(\TT^d)$, there exists a $v \in C^\infty_0(\TT^d)$ such that
$$
\Delta v = f - \fint_{\TT^d} f.
$$
And we denote $v$ by $\Delta^{-1}f$. Note that if $f\in C^\infty_0(\TT^d) $, then by rescaling  we have 
$$
\Delta^{-1} \big( f(\sigma \cdot )  \big) = \sigma^{-2}  v(\sigma \cdot ) \quad \text{ for $\sigma \in \NN$.}
$$

We recall the following antidivergence operator $\mathcal{R}$ introduced in \cite{MR3090182}.

\begin{definition}
$\mathcal{R} : C^\infty(\TT^d ,\RR^d) \to C^\infty(\TT^d, \mathcal{S}^{d \times d }_0)$ is defined by
\begin{equation}\label{eq:appendix_R_def}
(\mathcal{R} v )_{ij} =  \mathcal{R}_{ijk} v_k
\end{equation}
where
$$
 \mathcal{R}_{ijk} =  \frac{2-d}{d-1}\Delta^{-2} \p_i \p_j \p_k -  \frac{1}{d-1}\Delta^{-1} \p_k \delta_{ij} + \Delta^{-1} \p_i \delta_{jk} + \Delta^{-1} \p_j \delta_{ik} .
$$
\end{definition}

It is clear that $\mathcal{R} $ is well-defined since $\mathcal{R}_{ijk}$ is symmetric in $i,j$ and taking the trace gives
\begin{align*}
\Tr    \mathcal{R} v&=  \frac{2-d}{  d -1} \Delta^{-1}    \partial_k v_k + \frac{-d}{d-1} \Delta^{-1} \partial_k   v_k  +   \Delta^{-1}   \partial_k v_k + \Delta^{-1}   \partial_k v_k\\
&=(\frac{2-d}{  d  -1} + \frac{-d}{d-1}+  2) \Delta^{-1}  \partial_k v_k=0.
\end{align*}

By a direct computation, one can also show that
$$
\D (\mathcal{R} v  ) = v - \fint_{\TT^d} v  \quad \text{for any $v \in C^\infty(\TT^d ,\RR^d)$}
$$
and
\begin{equation}\label{eq:appendix_R_2}
\mathcal{R} \Delta v   = \nabla v + \nabla v ^T  \quad \text{for any divergence-free $v \in C^\infty(\TT^d ,\RR^d)$}.
\end{equation}

We can show that $ \mathcal{R}$ is bounded on  $L^{p} (\TT^d) $ for any $1\leq p \leq \infty$.
\begin{theorem}\label{thm:bounded_R}
Let $1 \leq p \leq \infty$. For any vector field $f \in C^\infty(\TT^d,\RR^d)$, there holds
$$
\| \mathcal{R} f \|_{L^{p}(\TT^d )} \lesssim \|  f \|_{L^{p}(\TT^d )}.
$$

In particular, if $f \in C^\infty_0(\TT^d,\RR^d)$, then
$$
\| \mathcal{R} f(\sigma \cdot ) \|_{L^{p}(\TT^d )} \lesssim \sigma^{-1}\|  f \|_{L^{p}(\TT^d )} \quad \text{for any $\sigma \in \NN$}.
$$
\end{theorem}
\begin{proof}
Once the first bound is established, the second bound follows from the definition of $ \mathcal{R}$. It suffices to only consider $f  $ with zero mean since $\mathcal{R} (C) =0$ for any constant $C$. 
Then we only need to show that the operator 
$$
\Delta^{-1} \Delta^{-1} \p_i \p_j \p_k
$$ is bounded on $L^p(\TT^d)$ for $1\leq p \leq \infty$ since the argument applies also to $\Delta^{-1} \p_i $.

When $1 < p < \infty$, this follows from the boundedness of the Riesz transforms and the Poincare inequality.

When $p= \infty$, the Sobolev embedding $W^{1,d+1}(\TT^d) \hookrightarrow L^{\infty}(\TT^d )$ implies that
$$
\|\Delta^{-1} \Delta^{-1} \p_i \p_j \p_k f \|_{L^\infty(\TT^d) } \lesssim \|\Delta^{-1} \Delta^{-1} \p_i \p_j \p_k f \|_{  W^{1,d+1} (\TT^d) } \lesssim \| f  \|_{L^{d + 1}(\TT^d) }\leq   \| f  \|_{L^{\infty}  (\TT^d) }
$$
where we have used the boundedness of the Riesz transforms once again.

When $p= 1$,  one can use a duality approach and use the boundedness in $L^\infty$ since integrating by parts yields
$$
\langle \Delta^{-1} \Delta^{-1} \p_i \p_j \p_k f ,\varphi \rangle = -\langle  f ,\Delta^{-1} \Delta^{-1} \p_i \p_j \p_k \varphi \rangle  \quad \text{ if $ \varphi \in C^\infty_0(\TT^d)$.}
$$

\end{proof}

\subsection{Bilinear antidivergence $\mathcal{B}$}

We can also introduce the bilinear version $\mathcal{B} : C^\infty(\TT^d, \RR^d) \times C^\infty(\TT^d   ,\RR^{d\times d} ) \to C^\infty(\TT^d, \mathcal{S}^{d \times d}_0) $ of $\mathcal{R}$. This bilinear antidivergence $\mathcal{B}$ allows us to gain derivative when the later argument has zero mean and a small period.

Let
\begin{equation*}
( \mathcal{B}( v,A ))_{i j}  = v_l \mathcal{R}_{ijk}A_{lk} - \mathcal{R}( \p_i v_l \mathcal{R}_{ijk}A_{lk} ) 
\end{equation*}
or by a slight abuse of notations
$$
  \mathcal{B}( v,A )  = v  \mathcal{R} A  - \mathcal{R}( \nabla v  \mathcal{R} A   ) .
$$

\begin{theorem}\label{thm:bounded_B}
Let  $1 \leq p \leq \infty$. For any $v \in C^\infty(\TT^d,\RR^d)$ and $A \in C^\infty_0(\TT^d,\RR^{d\times d})$, 
\begin{equation} \label{eq:div_B_tensor}
\D( \mathcal{B}(v , A)  )  =  v   A   - \fint_{\TT^d} vA ,
\end{equation}
and
$$
\|  \mathcal{B} (v ,A) \|_{L^{p}(\TT^d )} \lesssim \|   v \|_{C^{1}(\TT^d )} \|   \mathcal{R} A \|_{L^{p}(\TT^d )}.
$$
\end{theorem}
\begin{proof}
A direct compuation gives
\begin{align*}
\D( \mathcal{B}(v , A)  ) & = \p_j v_l \mathcal{R}_{ijk}A_{lk} +   v_l \p_j \mathcal{R}_{ijk}A_{lk} - \D \mathcal{R}( \p_i v_l \mathcal{R}_{ijk}A_{lk} ) \\
&=     v_l  A_{il}    + \fint \p_i v_l \mathcal{R}_{ijk}A_{lk}
\end{align*}
where we have used the fact that $A$ has zero mean and $\mathcal{R}$ is symmetric.

Integrating by parts, we have
\begin{align*}
 \fint \p_i v_l \mathcal{R}_{ijk}A_{lk} =  -\fint  v_l \p_i \mathcal{R}_{ijk}A_{lk} =-\fint  v_l A_{lj},
\end{align*}
which implies that
\begin{align*}
\D( \mathcal{B}(v , A)  )   
 =     v A -       \fint v A.
\end{align*}

The second estimate follows immediately from the definition of $\mathcal{B}$ and Theorem~\ref{thm:bounded_R}.

\end{proof}

\subsection*{Acknowledgement}
AC was partially supported by the NSF grant DMS--1909849. The authors are grateful to the anonymous referees for very helpful comments.

\bibliographystyle{alpha}
\bibliography{Sharp_NSE} 
%\nocite{*}

\end{document}